\documentclass[a4paper, 11pt, reqno]{amsart} % Font size (can be 10pt, 11pt or 12pt) and paper size (remove a4paper for US letter paper)
\usepackage[utf8]{inputenc}
\usepackage[numbers]{natbib}
\usepackage[pdftoolbar=true, pdfmenubar=true, pdftitle={}, pdfsubject={}, pdfauthor={}, pdfkeywords={}, pdfcreator={}, pdfproducer={}, bookmarks=false, bookmarksopenlevel=section]{hyperref}
\usepackage[paper=a4paper,left=25mm,right=25mm,top=25mm,bottom=25mm]{geometry}

\usepackage[babel]{csquotes}

\usepackage[english]{babel}	 
\usepackage{graphicx} % Required for including pictures
\usepackage{wrapfig} % Allows in-line images
\usepackage{amsmath}
\usepackage{amsthm}		
\usepackage{eucal}
\usepackage{amsfonts, dsfont}										
\usepackage{amssymb}	
\usepackage{bbm}										
\usepackage{braket}	
\usepackage{enumitem}

\numberwithin{equation}{section}

\newtheorem{theorem}{Theorem}[section]
\newtheorem{lemma}[theorem]{Lemma}

\newtheorem{proposition}[theorem]{Proposition}
\newtheorem{corollary}[theorem]{Corollary}
\newtheorem{definition}[theorem]{Definition}
\newtheorem{cond}[theorem]{Condition}

\theoremstyle{remark}
\newtheorem{remark}[theorem]{Remark}

\makeatletter
\renewcommand\@biblabel[1]{\textbf{#1.}} % Change the square brackets for each bibliography item from '[1]' to '1.'
\usepackage{color}
\DeclareMathOperator{\tr}{tr}
\DeclareMathOperator{\dom}{dom}

\DeclareMathOperator{\dist}{dist}

\DeclareMathOperator{\Op}{Op}
\newcommand{\conjugate}[1]{\ensuremath{\mkern 1.5mu\overline{\mkern-1.5mu #1 \mkern-1.5mu}\mkern 1.5mu}}

\newcommand{\chiab}{\chi_{(-\alpha,\alpha)}}

\begin{document}
%----------------------------------------------------------------------------------------
%	TITLE
%----------------------------------------------------------------------------------------

\title{\textbf{Formulas of Szeg\H o type for the 
periodic Schr\"odinger operator}} % Subtitle

\author{\textsc{Bernhard Pfirsch \& Alexander V. Sobolev}} % Author
%\\{\textit{University College London}}} % Institution


\date{\today} % Date

\newcommand{\Addresses}{{% additional braces for segregating \footnotesize
  \bigskip
  %\footnotesize

\textsc{Department of Mathematics,
University College London,
Gower Street,
London
WC1E 6BT, UK.}\par\nopagebreak
  \textit{e-mail:} \texttt{bernhard.pfirsch.15@ucl.ac.uk},\ \ \texttt{asobolev@math.ucl.ac.uk}
%  \medskip
}}
\begin{abstract}
We prove asymptotic formulas of Szeg\H o type for the periodic Schrödinger operator $H=-\frac{d^2}{dx^2}+V$ in dimension one. Admitting fairly general functions $h$ with $h(0)=0$, we study the trace of the operator $h(\chi_{(-\alpha,\alpha)}\chi_{(-\infty,\mu)}(H)\chi_{(-\alpha,\alpha)})$ and link its subleading behaviour as $\alpha\to\infty$ to the position of the spectral parameter $\mu$ relative to the spectrum of $H$.
\end{abstract}

\maketitle % Print the title section

%----------------------------------------------------------------------------------------
%	ABSTRACT AND KEYWORDS
%----------------------------------------------------------------------------------------

%----------------------------------------------------------------------------------------
%	ESSAY BODY
%----------------------------------------------------------------------------------------
%\tableofcontents
\section{Introduction}

The classical Szeg\H o formula (see \cite{Szego1952}) describes 
the determinant of the truncated Toeplitz 
matrix as the truncation parameter tends to infinity, 
we refer to survey \cite{Kras} for discussion and further references. 
Our interest is closer to the continuous variant of this problem, 
i.e. to truncated Wiener-Hopf operators. 
Let $I\subset \mathbb R$ be a finite (open) interval, and let 
$a = a(\xi), \xi\in\mathbb R$ 
be a bounded, in general complex-valued function, 
which we call \textit{symbol}. By  
the \textit{truncated Wiener-Hopf operator} 
%with the scaling parameter 
%$\alpha>0$ 
we understand the operator of the form 
\begin{align*}
W(a; I) = \chi_{I} \mathcal F^* a \mathcal F\chi_{I},
\end{align*} 
where $\mathcal F:{\rm L}^2(\mathbb R)\mapsto {\rm L}^2(\mathbb R)$ 
is the unitary Fourier transform, and $\chi_{I}$ is the indicator of the interval 
$I$. There is a vast literature studying the behaviour of the trace
\[
\tr h(W(a; \alpha I))
\]
 with 
a test function $h$, 
as the scaling parameter $\alpha$ tends to infinity. 
Assuming for simplicity 
that $h$ is continuous, 
one can claim that the above trace is 
finite if $h(0) = 0$ and the function $a$ decays sufficiently 
fast at infinity. 
We do not intend to give 
an extensive survey of known results, 
but only mention that, under the assumption that 
the functions $a$ and $h$ are smooth, one can find a complete asymptotic expansion of this trace 
in powers of $\alpha^{-1}$, see e.g. 
\cite{BuBu}, \cite{Widom1985}. Limited to two terms only, this expansion has the form
 \begin{align}\label{area:eq}
\tr h(W(a; \alpha I))
=   
 \frac{\alpha}{2\pi} |I| \int h\bigl(a(\xi)\bigr) d\xi 
 + \mathcal B + O(\alpha^{-1}), \ \alpha\to\infty,
 \end{align}
with an explicitly computable 
coefficient $\mathcal B = \mathcal B(a; h)$, independent of the interval $I$. Note that \cite{Widom1985} contains even the multidimensional version 
of the result.

In this paper we do not need the precise 
value of $\mathcal B$, since our main concern is the
case 
of a non-smooth symbol $a$.
Assume for the sake of discussion that $a = \chi_J$ 
where $J\subset \mathbb R$ 
is a bounded interval,  and that $h$ is a ${\rm C}^\infty$-function 
such that $h(0) = 0$. 
Then the results of 
\cite{LandauWidom}, and \cite{Widom1981} imply the asymptotic formula
\begin{align}\label{zero:eq}
\tr h(W(\chi_J; \alpha I))
=  \frac{\alpha}{2\pi} h(1) |I| |J|  
+ \log(\alpha) \mathcal W(h)+ o(\log(\alpha)), \ \alpha\to\infty,
\end{align}
with a coefficient $\mathcal W(h)$ independent 
of the intervals $I$ and $J$, see \eqref{whc} for the definition. 
Thus, one observes that 
the first term on the right-hand side is the same as in 
\eqref{area:eq}, but the second one exhibits a 
behaviour different 
from \eqref{area:eq}. 
The multidimensional generalization of this result, 
even with more general 
discontinuous symbols $a$ 
was obtained in \cite{Sobolev2010}, \cite{Sobolev2013}. 
Further extension to non-smooth 
functions $h$ was done in 
\cite{LeschkeSobolevSpitzer2014}, 
\cite{Sobolev2016}, \cite{Sobolev2016a}. 
The formula \eqref{area:eq} is a continuous analogue of 
the second-order Szeg\H o limit theorem, see \cite{Szego1952}, so 
we loosely refer to \eqref{area:eq} and \eqref{zero:eq} as 
\textit{Szeg\H o formulas}, or \textit{formulas of Szeg\H o type}. 
It is clear that under the condition 
$h(0) =  h(1) = 0$ the leading term in 
\eqref{zero:eq} vanishes, and 
the formula takes the form 
\begin{align}\label{en:eq}
\tr h(W(\chi_J; \alpha I))
= \log(\alpha) \mathcal W(h)+ o(\log(\alpha)), \ \alpha\to\infty,\ 
\textup{if} \ h(0) = h(1) = 0.
\end{align}
The increased recent interest in the asymptotic results of the described type 
with possibly non-smooth functions $h$ is 
partly due to their connection with the study of 
the \textit{bipartite entanglement entropy (EE)}, see e.g. 
\cite{GioevKlich2006}, \cite{Helling2009}, \cite{LeschkeSobolevSpitzer2014}, 
\cite{LeschkeSobolevSpitzer2016}. 
For instance, the formula 
\eqref{zero:eq}, used with the function 
\begin{align}\label{vneumann:eq}
\eta_1(t) = -t\log t - (1-t) \log(1-t), \ t\in [0, 1], 
\end{align} 
which is not differentiable at 
the endpoints of the interval $[0,1]$, 
would describe the scaling asymptotics of the 
von Neumann EE for free fermions in the Fermi 
sea $J$ at zero temperature, 
see \cite{Klich2006}, 
\cite{GioevKlich2006}. 
%Strictly speaking, the choice \eqref{vneumann:eq} 
%produces the \textit{von Neumann} EE, 
The function \eqref{vneumann:eq} is just one representative of the family 
\begin{align}\label{renyi:eq}
\eta_\gamma(t) = \frac{1}{1-\gamma} \log\big[
t^\gamma+(1-t)^\gamma
\big],\ t\in [0,1], 
\end{align}
with $\gamma>0$, where $\eta_1$ is defined as 
the limit of $\eta_\gamma$ as $\gamma\to 1$, $\gamma\not = 1$. 
Picking $h = \eta_\gamma$ one obtains 
from \eqref{zero:eq} the asymptotics of the 
\textit{$\gamma$-R\'enyi} EE, 
see e.g. \cite{LeschkeSobolevSpitzer2014}. Due to the condition 
$\eta_\gamma(0) = \eta_\gamma(1) = 0$, formula \eqref{en:eq} applies and the EE behaves 
as $\log(\alpha)$ as $\alpha\to \infty$. 

Let us remark at this point that 
there is an extensive physics literature on the 
topic of EE. However, we do not enter a detailed 
discussion of it in the context of this paper. 
For the interested reader we refer to general reviews 
\cite{Das2007}, \cite{Amico2008}, 
\cite{Calabrese2009}, 
\cite{LatorreRiera}, 
\cite{Laflorencie2016}, on the importance
of EE in the study of black holes, 
condensed matter systems and quantum information theory.
\vspace{0.5cm}

Having in mind the application to EE, 
a natural generalization of discussed questions 
is to move from 
free fermions to 
fermions in an external field. In mathematical terms that amounts to studying the trace of the operator 
\begin{align}\label{trH:eq}
h\big(\chi_{\alpha I}a(H)\chi_{\alpha I}\big),
\end{align}
where $H$ is some general self-adjoint one-particle Hamiltonian.  
Such an analysis for ergodic Hamiltonians $H$ was 
conducted in \cite{KirschPastur2014}, 
\cite{PasturSlavin2014}, \cite{ElgartPastur2016}, 
including multidimensional results. 
In this new 
setting a number of new and rather unexpected effects 
emerge. To give just one example, as follows from \cite{ElgartPastur2016}, 
the EE for Fermions at zero temperature in a disordered one-dimensional 
medium remains bounded as $\alpha\to\infty$, 
in contrast to the free case, mentioned above. 
 
Our objective in the present paper is to obtain formulas of Szeg\H o type for 
the operator 
\eqref{trH:eq}  
with the function $a = \chi_{(-\infty, \mu)}$, 
and with
$H$ being the Schr\"odinger operator
with a periodic potential in dimension one, i.e.  
$H = -d^2/dx^2 + V$, where $V$ is a real-valued periodic function. 
Without loss of generality we assume that the period equals $2\pi$. The parameter $\mu$ is naturally interpreted as 
the \textit{Fermi energy}.
 
In order to describe the results on 
an informal level, assume for simplicity that 
$h$ is a function such that $h(0) = h(1) = 0$. As stated in the 
the main theorem (see Theorem \ref{theorem}), 
the asymptotics 
depend on the position of the 
Fermi energy relative to the bands of the spectrum $\sigma(H)$. 
If $\mu$ is in a spectral 
gap then the trace remains bounded as $\alpha\to\infty$. If however, 
$\mu$ is inside a spectral band, then, somewhat surprisingly, the asymptotics 
are exactly as in the case $V\equiv 0$, i.e.  described by the formula 
\eqref{en:eq}.  The multi-dimensional case will be addressed in 
a subsequent publication. 

The methods used to prove Theorem \ref{theorem} 
may be seen as an extension of those presented in \cite{LandauWidom}, with 
some non-trivial modifications 
due to the presence of the periodic potential $V$. One important 
difference is that the 
reflection symmetry and continuous translational invariance of the Hamiltonian, central for the argument in 
\cite{LandauWidom}, are replaced with the 
invariance with respect to the lattice translations. 
Here we rely on the properties of Bloch functions associated with the 
operator $H$.   
Another new element is the extended choice of the test function $h$, 
see Condition \ref{h:cond}. In our paper the function $h=h(t),\ t \in [0, 1]$, 
is allowed to be piece-wise continuous, 
although we require $h$ to be H\"older-continuous at the endpoints  
$t = 0,1$. This is made possible by adjusting 
the Schatten-von Neumann class estimates 
 for pseudo-differential operators with 
discontinuous symbols, obtained in \cite{Sobolev2014}.
Consequently, functions such as \eqref{renyi:eq} are covered by our result. 

A few comments on the structure of this paper are in order. 
We begin with recalling some fundamental properties of one-dimensional 
periodic Schr\"odinger operators (cf. Section \ref{prelim:sect}) 
and afterwards state our results in Section \ref{sect:results}. 
The very first step of the proof is to obtain an 
approximation of the kernel of the spectral projection $P_\mu$ 
in terms of Bloch eigenfunctions corresponding to the Fermi energy $\mu$, 
see Section \ref{sect:expintkern}.  
Section \ref{sect:elemtrnest} contains some elementary trace class estimates, 
similar to the ones obtained in \cite{LandauWidom}. Here we also introduce an averaging procedure 
for a particular type of integral operators 
(see sub-Section \ref{almperbymean:subsec}) that allows 
us to average out the precise dependence on the Bloch 
eigenfunctions at Fermi energy $\mu$. This is sufficient to 
prove Theorem \ref{theorem} for polynomial functions $h$, see Section 
\ref{sect:proofpol} and sub-Section \ref{ap:subsect}. 
As mentioned earlier, 
the extension to non-smooth functions calls for more advanced bounds 
in Schatten-von Neumann classes. These bounds are collected in 
Section \ref{sect:schattenest}. The extension to non-smooth 
$h$, i.e. the closure of the asymptotics from the polynomial $h$, 
is implemented in sub-Section  \ref{closure:subsect}.

To conclude the introduction let us fix some general notation. 
If $f$, $g$ are non-negative functions, we write
$f\ll g$ if $f\leq Cg$ for some 
constant $C\geq 0$. This constant may depend on the potential 
$V$ but does not depend on the dilation parameter 
$\alpha$. To avoid confusion we sometimes 
make explicit comments 
on the nature of (implicit) constants in the bounds.

For a set $I\subset \mathbb R$ the notation $I^\circ$ 
is used for the set of all interior points of $I$ 
and its Lebesgue measure is denoted by 
$|I|$. In many situation 
(e.g. for the intervals $I$, $J$, $K$ in Section \ref{sect:elemtrnest}) 
it will not matter whether considered intervals are open, semi-open or closed. 
Whenever this is the case we shall use open intervals only.
  
\textit{Acknowledgement.} 
The authors are grateful to E. Korotyaev for his advice 
on periodic Schr\"odinger operators.

\section{Preliminaries}\label{prelim:sect}

We consider a periodic Schrödinger operator
\begin{align*}
{H}=-\frac{d^2}{dx^2}+V(x), \ \dom(H)={\rm H}^2(\mathbb{R}),
\end{align*}
in dimension $1$. 
More precisely, let the potential $V$ be a real-valued 
$2\pi$-periodic ${\rm L}^2_{\textup{\tiny loc}}$-function, 
so that the operator $H$ 
is self-adjoint 
on ${\rm H}^2(\mathbb R)$. 
For $\mu\in\mathbb{R}$, we introduce the notation 
$P_\mu:=\chi_{(-\infty, \mu)}(H)$ for the spectral 
projection of $H$. We shall use $\alpha>0$ as a 
dilation parameter and write $\chi_I$ for the 
indicator  
function of the interval $I\subset \mathbb R$ as 
well as for the corresponding multiplication 
operator on ${\rm L}^2(\mathbb{R})$. 
For an appropriate choice of the function $h$, we 
are interested in an asymptotic formula for the trace 
\begin{align}\label{bal:eq}
\tr h(B_{\alpha, \mu}),\  
B_{\alpha, \mu} = \chi_{(-\alpha, \alpha)}
P_\mu \chi_{(-\alpha, \alpha)},
\end{align}
as $\alpha\to\infty$. 

We heavily rely on 
the standard Floquet-Bloch theory for periodic operators, 
see e.g. \cite{ReedSimon4}, \cite{TODE}. 
In particular, we make use of the Floquet-Bloch-Gelfand transform
\begin{align*}
U:{\rm L}^2(\mathbb{R})\longrightarrow 
{\rm L}^2\big( \mathbb T, {\rm L}^2(0,2\pi)\big), 
\mathbb T = \mathbb R/\mathbb Z.
\end{align*} 
For Schwartz class functions 
or ${\rm L}^2(\mathbb{R})$-function with compact support, 
it is given by
\begin{align*}
(U\psi)(x, k):=\sum\limits_{\gamma\in 
2\pi\mathbb{Z}}e^{-ik\gamma}\psi(x+\gamma),\ k\in \mathbb T,\ x\in[0,2\pi].
\end{align*}
The operator $U$ is easily checked to be isometric, 
and hence it extends by continuity 
as a unitary operator to the entire 
${\rm L}^2(\mathbb{R})$. 
Under $U$ the periodic Schrödinger operator $H$ 
transforms into the direct integral
\begin{align*}
UHU^\ast = \int\limits_{\mathbb T}^{\oplus} H(k) dk,
\end{align*}
with self-adjoint fibres 
\begin{align}
H(k)= &\ -\frac{d^2}{dx^2}+V(x),\notag\\[0.2cm]
\dom\big(H(k)\big) = &\ \lbrace f\in {\rm H}^2(0,2\pi):
 f(2\pi)=e^{2\pi i k}f(0), \ f'(2\pi)=e^{2\pi i k}f'(0)\rbrace,\label{dom:eq}
\end{align} 
that are well-defined for $k\in \mathbb T$. 
It is well-known that each fibre operator $H(k)$ has compact resolvent and, therefore, a discrete spectrum that consists of eigenvalues 
$\lambda_j(k),\ j = 1, 2, \dots,$ labelled in ascending order 
counting multiplicity. 
Denote the corresponding normalized eigenfunctions by 
$\phi_j(k) = \phi_j(\ \cdot\ , k)\in {\rm H}^2(0, 2\pi)$, 
$j= 1, 2, \dots$. 
 
It is clear that for all $k\in \mathbb R$ the functions 
\begin{align}\label{ejdef}
{\rm e}_j(x, k):= e^{-i k x} \phi_j(x, k)
\end{align}
and their derivatives $d{\rm e}_j/dx$ can be extended to all $x\in\mathbb R$  
as $2\pi$-periodic functions,   
which induces  a corresponding extension of 
$\phi_j(\ \cdot\ , k)$. 
Using the eigenfunctions $\phi_j(k)$ 
we can write out the kernel $P_\mu(x, y)$ of the projection 
$P_\mu$:
\begin{align}\label{proj:eq}
P_\mu(x, y) = \sum_j \int\limits_{\mathbb T} 
\chi_{(-\infty, \mu)}(\lambda_j(k)) 
\phi_j(x, k) \conjugate{\phi_j(y, k)} dk.
\end{align}
%
%
%A detailed presentation of the Floquet-Bloch 
%decomposition is given for example in 
%\cite{ReedSimon4}. 
In the next proposition we summarize the properties of 
the functions $\phi_j(k)$ and eigenvalues $\lambda_j(k)$ that we use 
further on. The points $k=0$ and $k=\frac{1}{2}$ will play a special role, so it makes sense to introduce temporarily the notation
\begin{align*}
\mathbb T_0 = \mathbb T\setminus\bigg(\{0\}\cup\bigg\{\frac12\bigg\}\bigg).
\end{align*}
 
\begin{proposition}\label{basic:prop}
Let $H(k), k\in\mathbb T,$ be as defined above. Then 
\begin{enumerate}
\item
 For every $k\in \mathbb T$ 
 the operators $H(k)$ and $H(-k)$ are antiunitarily equivalent under 
 complex conjugation. In particular, 
 $\lambda_j(k) = \lambda_j(-k)$ for all $j = 1, 2, \dots$.
\item 
The eigenfunctions $\phi_j(\ \cdot\ , k), j = 1, 2, \dots,$
can be chosen to be analytic in $k\in \mathbb T_0$, and such that 
$\phi_j(-k) = \conjugate{\phi_j(k)}$, $k\in\mathbb T_0$. 
\item
The eigenvalues $\lambda_j(k),\ j = 1, 2, \dots,$ are even 
continuous functions of $k\in \mathbb T$. 
These eigenvalues are simple and analytic on $\mathbb T_0$. 
\item\label{incr:item}
For $j$ odd (resp. even) each $\lambda_j(\ \cdot\ )$ is strictly increasing (resp. 
decreasing) on $(0, \frac{1}{2})$.
\end{enumerate}
\end{proposition}

\vskip 0.5cm

Let  
\begin{align}\label{kj:eq}
k_j =
\begin{cases}
0,\ \ j \ \textup{odd},\\[0.2cm]
\frac{1}{2},\ \ j\ \textup{even}.
\end{cases}
\end{align}
Denote
\begin{align*}
\mu_j = \lambda_j(k_j),\ \nu_j = \lambda_j\bigg(k_j+\frac{1}{2}\bigg),\ 
\ \ \ 
\sigma_j = [\mu_j, \nu_j],\ j = 1, 2, \dots.
\end{align*}
%and 
%\begin{align*}
%\sigma_j = 
%\begin{cases}
%[\mu_j, \nu_j],\ j\  \textup{even},\\
%[\nu_j, \mu_j],\ j \ \textup{odd}.
%\end{cases}
%\end{align*}
The spectrum $\sigma(H)$ of $H$ is represented 
as the union of \textit{spectral bands} $\sigma_j$:
\begin{align*}
\sigma(H) = \bigcup\limits_{j=1}^\infty \, \sigma_j.
\end{align*}
It follows from Proposition \ref{basic:prop}(\ref{incr:item}) 
that the bands $\sigma_j$ are non-degenerate, i.e. 
$|\sigma_j|>0$ for every $j=1, 2, \dots.$  
Introduce the counting function of $H(k)$: 
\begin{align*}
N(\mu, k) = \# \lbrace j: \lambda_j(k)< \mu \rbrace, \ \mu\in\mathbb{R}, \ 
k\in\mathbb T,
\end{align*}
and \textit{the (integrated) density of states:}
\begin{align}\label{density:eq}
N(\mu; H)=\frac{1}{2\pi}\underset{\mathbb T}\int N( \mu, k)\, dk.
\end{align}
In view of Proposition \ref{basic:prop}(\ref{incr:item}) again,
the function \eqref{density:eq} is continuous.  
The definition \eqref{density:eq} agrees 
with the standard definition of the density of states which is given via the Hamiltonian with Dirichlet 
boundary condition on a large cube, see e.g. 
\cite[Theorem 4.2]{Sh1} or \cite[Ch. XIII]{ReedSimon4}.

Note also that the spectral bands cannot overlap, but they may touch. 
This situation is our main concern in the next proposition.

\begin{proposition} \label{basic1:prop}
Let $\lambda_j=\lambda_j(k)$, 
$\phi_j = \phi_j(k)$ be as described in Proposition \ref{basic:prop}. 
Then 
\begin{enumerate}
\item
If for some $j$ 
the bands $\sigma_{j-1}$ and $\sigma_j$ are separated from each other, 
i.e. $\nu_{j-1}< \mu_j$, then the eigenvalues $\lambda_{j-1}(\ \cdot\ )$, 
$\lambda_{j}(\ \cdot\ )$ 
and eigenfunctions $\phi_{j-1}(x, \ \cdot\ )$, 
$\phi_{j}(x, \ \cdot\ )$ 
are analytic in $k $ in a neighbourhood 
of $k_j$, 
for each $x\in \mathbb R$. Furthermore, 
the functions 
$\phi_{j-1}(\ \cdot\ , k_j)$ and $\phi_j(\ \cdot\ , k_j)$ 
are real-valued. 
\item\label{touch:item}
If for some $j$ we have $\nu_{j-1} = \mu_j$, i.e. 
$\lambda_{j-1}(k_j) = 
\lambda_{j}(k_j)$, 
then 
in a neighbourhood of $k_j$, the eigenvalues 
$\lambda_{j-1}$ and $\lambda_{j}$, and 
the eigenfunctions 
$\phi_{j-1}(x, \ \cdot\ )$ 
and 
$\phi_{j}(x, \ \cdot\ )$ 
are analytic continuations of each other.
Moreover, $\lambda_l'(k_j\pm) \not = 0$, 
$\phi_l(k_j-) = \conjugate{\phi_l(k_j+)}$ , 
and the limits $\phi_l(k_j-)$ and $\phi_l(k_j+)$ are 
mutually ${\rm L}^2$-orthogonal 
for $l = j-1, j$. 
\end{enumerate}
\end{proposition}

\vskip 0.5cm

Although Propositions \ref{basic:prop} and 
\ref{basic1:prop} are well-known, we need to make some comments. 
The analyticity on $\mathbb T_0$ 
in Proposition \ref{basic:prop} is a straightforward consequence of 
the analytic perturbation theory, see \cite{Kato} or \cite{ReedSimon4}. 
The analyticity of $\phi_j(k)$ and $\lambda_j(k)$   
in the vicinity of points $k_j$ in 
Proposition \ref{basic1:prop} is a more subtle fact, and it 
follows from abstract theorems 
\cite[Ch. II, Theorems 1.9, 1.10]{Kato}
and \cite[Theorems XII.12, XII.13]{ReedSimon4}. 
In the context of the periodic operators the analytic properties of 
eigenfunctions $\phi_j$ are described 
in \cite{Fir1} and \cite{Fir2}. 
Also, the relation $\lambda_l'(k_j\pm)\not = 0$ from Proposition \ref{basic1:prop}(\ref{touch:item}) can be found e.g. in \cite[formula (3.1)]{Fir2}.

Assume again that $\nu_{j-1} = \mu_{j}$, 
i.e. the bands $\sigma_{j-1}$ and $\sigma_{j}$ have one 
common point. 
Since $\phi_{j-1}(k)$ and $\phi_{j}(k)$ are orthogonal for 
all $k\in \mathbb T_0$, and $\phi_{j-1}(k_j\pm) = \phi_{j}(k_j\mp)$, 
the functions $\phi_j(k_j-)$ and $\phi_j(k_j+)$ are mutually orthogonal, 
as claimed in Proposition \ref{basic1:prop}(\ref{touch:item}). 
In view of the identity $\phi_j(k_j-) = \conjugate{\phi_j(k_j+)}$, 
this implies that 
\begin{align}\label{meant:eq}
\int\limits_0^{2\pi} \phi_j(x, k_j\pm)^2 dx = 0,\ \ \textup{if}\ \ 
\nu_{j-1} = \mu_{j}.
\end{align}
It is natural to group the bands that have common points (i.e. touch) 
together. Suppose that 
the bands 
$\sigma_j, \sigma_{j+1},\ \dots, \sigma_{j+n-1}$  are of this type 
and that $\sigma_{j-1}\cap\sigma_j = \varnothing$, \ 
$\sigma_{j+n-1}\cap\sigma_{j+n} = \varnothing$. 
Thus the interval 
\begin{align}\label{properband:eq}
S = \bigcup\limits_{l=0}^{n-1} \sigma_{j+l} = [\mu_j, \nu_{j+n-1}],
\end{align}
is a ``genuine" spectral band. 
Sometimes we informally use this term, ``genuine", to distinguish 
the bands $\{\sigma_j\}$ and $S$.
Using this construction, we can somewhat simplify the description of 
the spectral structure of $H$ inside $S$. 
Indeed, define on $[k_j-n/2, k_j+n/2]$ the real-valued 
function  
\begin{align*}
\Lambda(k)
= &\ \lambda_{j+l}(k),\ k\in \biggl[k_j 
+ \frac{l}{2}, k_j+ \frac{l+1}{2}\biggr], \ l = 0, 1, \dots, n-1,\\[0.2cm]
\Lambda(k) = &\ \Lambda(2k_j-k),\ k\in 
\biggl[k_j-\frac{n}{2}, k_j\biggr].
\end{align*}
According to Propositions \ref{basic:prop} 
and \ref{basic1:prop}, the above function is analytic on the circle  
$n\mathbb T = \mathbb R/n\mathbb Z$, monotone increasing in 
$k\in [k_j, k_j + n/2]$, and symmetric in $k = k_j$. 
Note also that 
\begin{align}\label{infsup:eq}
\Lambda(k_j) = \mu_j>\nu_{j-1},
\ \Lambda(k_j+n/2) = \nu_{j+n-1} < \mu_{j+n}. 
\end{align}
In the same way, one defines the eigenfunction $\Phi(x, k)$ that 
incorporates all of the $\phi_{j+l}$'s, $l = 0, 1, \dots, n-1$:
\begin{align}
\Phi(k)
= &\ \phi_{j+l}(k),\ k\in \biggl[k_j + \frac{l}{2}, k_j+ \frac{l+1}{2}\biggr], \ l = 0, 1, \dots, n-1,\notag\\[0.2cm]
\Phi(k) = &\ \conjugate{\Phi(2k_j-k)},\ k\in 
\biggl[k_j-\frac{n}{2}, k_j\biggr].\label{symPhi:eq}
\end{align}
 Similarly to $\Lambda(\ \cdot\ )$, 
 the function $\Phi(x, \ \cdot\ )$ is analytic on 
the circle $n\mathbb T$. The functions $\Phi(\ \cdot\ , k_j)$ 
and $\Phi(\ \cdot\ , k_j+n/2)$, associated with the ends 
of the band $S$, are real-valued. It is also useful to 
define the function 
\begin{align}\label{Edef:eq}
E(x, k) = e^{-ix k}\Phi(x, k),
\end{align}
built out of the functions \eqref{ejdef} in the same way as $\Phi(k)$ 
is built out of $\phi_j(k)$'s. The functions $E(x, k)$ are analytic in 
$k\in \mathbb R$, 
and $2\pi$-periodic in $x\in\mathbb R$. 

 Of course, it may happen that \textit{all} bands 
starting with $\sigma_j$, %$\sigma_j, \sigma_{j+1},\dots$ 
touch. In this case the above construction still works and yields analytic functions $\Phi$, $\Lambda$, and $E$ on $\mathbb{R}$. To keep the notation simple we shall allow in the following $n=\infty$ and use the convention $\infty \mathbb{T}=\mathbb{R}$.

Using the functions 
$\Lambda$ and $\Phi$ we can write the spectral 
representation of the operator $H$ as follows. 
Let $P[S]$ be the spectral projection 
of $H$ corresponding to a band $S$ defined as in \eqref{properband:eq}. Then
\begin{align*}%\label{boldphi:eq}
UH P[S]U^* 
= \int\limits_{n\mathbb T} \Lambda(k) P[\Phi(k)]\, dk,
\end{align*} 
where 
$P[\psi]$ is the orthogonal projection in 
${\rm L}^2(0, 2\pi)$ on the span of the function $\psi\in {\rm L}^2(0, 2\pi)$. 
As a consequence, 
the formula \eqref{proj:eq} leads to the following formula for 
the kernel $P_\mu[S](x, y)$ of the projection 
$P_\mu[S]:= P_\mu P[S]$: 
\begin{align}\label{projS:eq}
P_\mu[S](x, y) = 
\int\limits_{k\in n\mathbb T: \Lambda(k) < \mu}\mkern-26mu \Phi(x, k)
\conjugate{\Phi(y, k)}\, dk.
\end{align}

Given the properties of the eigenfunction $\Phi(\ \cdot\ , k), j= 1, 2, \dots,$ 
for every $k\in n\mathbb T$, it belongs to the algebra ${\rm CAP}(\mathbb R)$ 
of  continuous 
almost-periodic functions on $\mathbb R$, 
which is defined as the closure of the span of exponentials 
$e^{ix\xi},\ \xi\in \mathbb R,$ in the ${\rm L}^\infty$-norm. 
For any $f\in {\rm CAP}(\mathbb R)$ 
the \textit{almost-periodic mean} 
\begin{align*}
\mathcal{M}(f):=\lim\limits_{T\to\infty}(2T)^{-1}\int\limits_{-T}^T dt\,f(t)
\end{align*} 
is well-defined.
For an introduction to almost periodic functions and their 
properties we refer to \cite{Sh1} or \cite{Shubin1978}. 
 
For the future use we need to evaluate some means
for the eigenfunctions $\Phi(k)$.

\begin{lemma}\label{mean:lem} 
Let $\Phi = \Phi(\ \cdot\ , k)$ be the eigenfunction associated with the band $S$, 
see \eqref{properband:eq}. Then 
\begin{align}\label{mean:eq}
\mathcal M(|\Phi|^2) = \frac{1}{2\pi},\ \forall k\in n\mathbb T,
\end{align}
and 
\begin{align}\label{mean0:eq}
\mathcal M(\Phi^2) = 0,\ 
\forall k\not=k_j, k\not = k_j+n/2.
\end{align}

\end{lemma}

\begin{proof}
The function $\Phi$ 
is normalised in ${\rm L}^2(0,2\pi)$, 
whence $\mathcal M(|\Phi|^2) = (2\pi)^{-1}$, 
as claimed in \eqref{mean:eq}. 

To prove \eqref{mean0:eq},  suppose first that 
$2k\not\equiv 1\!\mod \mathbb Z$, 
so that  
$k\not = k_j\pm l/2$, $l = 0, 1, \dots, n$. 
We use the representation \eqref{Edef:eq}, so 
\begin{align*}
\mathcal M(\Phi^2) = \mathcal M(e^{2ikx}E^2).
\end{align*}
The function $w = E(\ \cdot\ , k)^2$ is continuous and 
$2\pi$-periodic. 
Picking an $\varepsilon>0$ we can 
approximate $w$ by trigonometric polynomials
\begin{align*}
p(x) = \sum_{s=-N}^N p_s e^{isx},
\end{align*}
so that 
$w = p + \tilde p$, where 
$\tilde p$ is a continuous periodic function such that 
$|\tilde p|<\varepsilon$. 
Let us find the mean for each 
component of the polynomial $p$ separately:
\begin{align*}
\int\limits_{-T}^T e^{2ikx + isx} dx 
= \left.\frac{e^{i(2k+s)x}}{i(2k+s)}\right|_{-T}^T,
\end{align*} 
which is bounded uniformly in $T$ for all $s = -N, -N+1, \dots, N$. Thus 
$\mathcal M(e^{2ikx}p) = 0$, and hence 
\begin{align*}
|\mathcal M(\Phi^2)| = |\mathcal M(e^{2ikx} \tilde p)| \le \varepsilon.
\end{align*}
As $\varepsilon>0$ is arbitrary, this entails that 
$\mathcal M(\Phi^2) = 0$, as required. 
 
The 
points $k = k_j\pm l/2$, $l =1, 2, \dots, (n-1)$ 
are exactly those, where the bands $\sigma_{j+l-1}$ and $\sigma_{j+l}$ 
touch. Thus the equality $\mathcal M(\Phi^2) = 0$ for these values of $k$ 
follows from \eqref{meant:eq}. This leads to \eqref{mean0:eq} again. 
\end{proof}

\section{Results}\label{sect:results}

Our main result concerns the trace defined in 
\eqref{bal:eq} with a test-function 
which satisfies the following condition. 

\begin{cond}\label{h:cond}
The function $h:[0, 1]\mapsto \mathbb C$ is piece-wise continuous, 
it is H\"older continuous at $t=0$ and $1$, and $h(0) = 0$.  
\end{cond}
 
For  a function $h$ 
satisfying Condition 
\ref{h:cond}, define the integral 
\begin{align}\label{whc}
\mathcal{W}(h) := 
\frac{1}{\pi^2}\int\limits_{0}^1\,
\frac{[h(t)-th(1)]}{t(1-t)} dt.
\end{align} 
The next theorem 
contains a Szeg\H o type formula for the operator 
$B_{\alpha, \mu}$ (see \eqref{bal:eq}), 
and it 
constitutes the main result of the paper.

\begin{theorem}\label{theorem} 
Suppose that 
$V\in {\rm C}^\infty(\mathbb R)$. 
Assume that the function $h$ satisfies Condition \ref{h:cond}. 
Then for any $\mu\in \big(\sigma(H)\big)^\circ$ 
we have the asymptotic formula
\begin{align}\label{trhb}
\tr[h(B_{\alpha,\mu})] = 
2\alpha h(1) N(\mu, H) + 
\log(\alpha)\mathcal{W}(h)+o(\log(\alpha)),\ 
\text{as}\ \alpha\to\infty.
\end{align}
If $\mu\notin \big(\sigma(H)\big)^\circ$, then
\begin{align}\label{trhb0:eq}
\tr[h(B_{\alpha,\mu})] = 
2\alpha h(1) N(\mu, H) +\mathcal O(1),\ 
\text{as}\ \alpha\to\infty.
\end{align}
Here, $N(\mu,H)$ denotes the integrated 
density of states for the operator $H$, defined in \eqref{density:eq}. 
\end{theorem}
  
\begin{remark}\label{lead:rem}
\begin{enumerate}
\item
The two terms in \eqref{trhb} are of different nature: 
the first one (linear in $\alpha$) 
depends on both the potential $V$ and the parameter $\mu$, 
but the second one ($\log$-term) 
is independent of $V$ or $\mu$, as long as $\mu$ remains an interior point of the 
spectrum $\sigma(H)$. 

\item 
The ways in which one finds the main term and the $\log$-term in \eqref{trhb} 
are completely different. To explain this, 
represent the function $h$ as the sum
\begin{align}\label{decomph}
h(t) = t h(1) + h_1(t),
\end{align}
so that the function $h_1$ satisfies Condition \ref{h:cond} and, in addition, 
$h_1(1) = 0$. 
The function $t h(1)$ 
is responsible for the first term in \eqref{trhb}. 
Indeed, 
according to \eqref{proj:eq},
\begin{align*}
\|B_{\alpha, \mu}\|_{\mathfrak S_1} 
= \tr B_{\alpha, \mu} = \sum\limits_j
\int\limits_{-\alpha}^{\alpha}\, \int\limits_{k\in \mathbb T: \lambda_j(k)<\mu} 
\!\!\! |\phi_j(x, k)|^2\, dk dx.
\end{align*}
Assume for simplicity that $\alpha$ is a multiple of $2\pi$.  
Since the $\phi_j$'s 
are normalized on $(0, 2\pi)$, by the definition \eqref{density:eq},  
we have 
\begin{align*}
\tr B_{\alpha, \mu} = 
\frac{\alpha}{\pi}\,\, \sum_j 
\int\limits_{k\in \mathbb T: \lambda_j(k)<\mu} \mkern-18mu
dk = 2\alpha N(\mu, H).
\end{align*} 
If $\alpha$ is not a multiple of $2\pi$,
then one easily checks, using the monotonicity of the trace in $\alpha$, that 
\begin{align}\label{balphabound:eq}
4\pi \left\lfloor \frac{\alpha}{2\pi}\right\rfloor
N(\mu, H)\le \tr B_{\alpha, \mu}\le 
4\pi \left\lceil \frac{\alpha}{2\pi}\right\rceil N(\mu, H),\ \forall \alpha >1.
\end{align}
Consequently, we conclude that
\begin{align*}
\tr B_{\alpha, \mu} = 2\alpha N(\mu, H) + \mathcal O(1),\ \alpha\to\infty.
\end{align*}
The study of $\tr h_1(B_{\alpha, \mu})$ is much more difficult, 
and the rest of the paper is focused on this task. 

\item
We point out that the function $h$ in Theorem \ref{theorem} 
is not required to be smooth, not even at the endpoints $t = 0, 1$. 
If we do assume that $h$ is differentiable at the endpoints, then 
the conditions on the potential $V$ can be reduced to 
$V\in {\rm L}^2_{\textup{\tiny loc}}(\mathbb R)$, as in 
Section \ref{prelim:sect}. This can be observed 
at the first step of the proof of Theorem \ref{theorem}, 
see Subsection \ref{closure:subsect}. The increased smoothness 
of $V$, i.e. the condition $V\in {\rm C}^\infty(\mathbb R)$ 
is required to handle the functions $h$ that are H\"older-continuous 
at $t = 0, 1$. To be precise, a finite smoothness of $V$,
depending on the H\"older exponent, 
would have been sufficient, but we do not go into 
these details to avoid excessive technicalities. 
\end{enumerate}
\end{remark}

\section{An Expansion of the Integral Kernel of the Spectral Projection}\label{sect:expintkern}
 
Let us temporarily assume that 
$\mu\in S$ where $S$ is a ``genuine" band of 
$\sigma(H)$ defined in \eqref{properband:eq}. Inspecting the formula \eqref{projS:eq}, 
we observe that the set $\{k: \Lambda(k)<\mu\}$ 
is the interval $(2k_j-\delta, \delta)$ where 
$\delta=\delta(\mu)\in [k_j, k_j+n/2]$ is the uniquely defined value such that 
$\Lambda(\delta) = \mu$. The following lemma provides a 
convenient expansion of $P_\mu[S]$ in powers of $|x-y|^{-1}$.

\begin{lemma}\label{exppmu:lem}
Let $\mu\in S$, where 
$S$ is the band defined in \eqref{properband:eq},  
and let $\delta=\delta(\mu)$ be as defined above. 
Then for all $x$, $y\in\mathbb{R}$ we have 
\begin{align}\label{kernel_mod0:eq}
P_\mu[S](x,y) 
= \Pi_\mu(x, y) + R_\mu[S](x,y),
\end{align}
where 
\begin{align}\label{leadpmu:eq}
\Pi_\mu (x, y)= \frac{\Phi(x, \delta)\conjugate{\Phi(y, \delta)} - 
\conjugate{\Phi(x, \delta)}\Phi(y, \delta)}{i(x-y)},
\end{align}
and
\begin{align}\label{boundrst}
R_\mu[S](x,y)=\mathcal{O}\big( (1+|x-y|^2)^{-1}\big),\ \forall 
x, y\in\mathbb R.
\end{align}
Moreover, $R_\mu[S](x,y)$, $P_\mu[S](x, y)$ 
and $\Pi_\mu(x, y)$ are continuous functions of 
$x, y\in\mathbb R$, and 
\begin{align}\label{pipS:eq}
|P_\mu[S](x, y)| + |\Pi_\mu (x, y)| = 
\mathcal{O}\big( (1+|x-y|)^{-1}\big),\ \forall 
x, y\in\mathbb R.
\end{align}
If $\mu\notin S^\circ$, then 
$P_\mu[S](x, y)$ 
is a continuous function of $x, y\in\mathbb R$, and it satisfies 
the bound 
\begin{align}\label{pmu_edgeS:eq}
P_\mu[S](x, y) = 
\mathcal{O}\big( (1+|x-y|^2)^{-1}\big),\ \forall 
x, y\in\mathbb R.
\end{align}
\end{lemma}

\begin{proof} 
Let us deduce the bound \eqref{pmu_edgeS:eq} first. 
Observe that if $\mu\notin S^\circ$, 
then either $P_\mu[S] = 0$ (if $\mu$ is below $S^\circ$), 
or $P_\mu[S] = P_{\mu_0}[S]$ where $\mu_0 = \nu_{j+n-1}$, i.e. 
$\delta(\mu_0) = k_j+n/2$ 
(if $S$ is bounded). 
In the first case the bound \eqref{pmu_edgeS:eq} is trivial. In the second case 
the function $\Phi(\delta)$ is real-valued, so that $\Pi_{\mu_0}(x, y) = 0$, 
and hence the bound \eqref{pmu_edgeS:eq} follows from \eqref{kernel_mod0:eq} and \eqref{boundrst}. 
  
It remains to prove the continuity and the bounds 
\eqref{boundrst} and \eqref{pipS:eq} for $\mu\in S$.   
Note that the kernel $P_\mu[S]$ (cf. \eqref{projS:eq}) is bounded uniformly 
in $x, y$, as $\Phi(x, k)$ is uniformly bounded
due to \eqref{Edef:eq}. 
It is also continuous in $x$, $y$. 
Furthermore, since $E(x, k)$ are continuous and periodic in $x$, the kernel 
\eqref{leadpmu:eq} is continuous and bounded 
by $|x-y|^{-1}$ 
for all $x, y: |x-y|\ge 1$. 
Due to the continuity of the derivative $\Phi_x$, 
the kernel \eqref{leadpmu:eq} is continuous and uniformly bounded 
for $|x-y|< 1$. 
As a consequence, 
$\Pi_\mu(x, y)$ satisfies \eqref{pipS:eq}, and 
the remainder $R_\mu[S](x, y)$ is continuous and uniformly bounded. 
Thus it remains to prove the bounds \eqref{boundrst} and \eqref{pipS:eq} 
for $P_\mu(x, y)$ with $|x-y|\ge 1$.
 
Using \eqref{Edef:eq}, we rewrite 
\begin{align}\label{calP:eq}
P_\mu[S](x,y)
=\int\limits_{2k_j-\delta}^{\delta}  \, e^{i k(x-y)} E(x, k)
\conjugate{E(y, k)} dk,
\end{align}
and integrate by parts to arrive at
\begin{align*}
P_\mu[S](x,y) 
= &\frac{e^{i\delta(x-y)}}{i(x-y)}
E(x, \delta) \conjugate{ E (y, \delta)}
-\frac{e^{i(2k_j-\delta)(x-y)}}{i(x-y)}
E(x, 2k_j-\delta) \conjugate{ E(y, 2k_j-\delta)}
+ R_\mu[S](x,y)
\end{align*}
with
\begin{align*}
R_\mu[S](x,y) 
= -\int\limits_{2k_j-\delta}^\delta \, 
\frac{e^{ik(x-y)}}{i(x-y)}\partial_k
\big( E(x, k) E(y, k) 
\big)dk.
\end{align*}
Due to \eqref{Edef:eq} and
the symmetry property \eqref{symPhi:eq} 
one obtains the representation \eqref{kernel_mod0:eq}. 
Another integration by parts for $R_\mu[S]$ gives 
\begin{align}
R_\mu[S](x, y)& = 
\left.\frac{e^{i k(x-y)}\partial_k 
\big(E(x, k) E(y, k) 
\big)}
{(x-y)^2}\right|_{2k_j-\delta}^\delta 
- \int\limits_{2k_j-\delta}^\delta \, 
\frac{e^{i k(x-y)}}{(x-y)^2}\partial_k^2
\big( E(x, k) E(y, k) 
\big)dk.
\end{align}
Hence, estimate \eqref{boundrst} 
follows from the fact that the functions $E$, $\partial_k E$, 
and $\partial_k^2 E$ are uniformly bounded. 
\end{proof}

Now Lemma \ref{exppmu:lem} may be used for each ``genuine" 
spectral band separately to get the 
corresponding expansion of the kernel $P_\mu(x, y)$.

\begin{lemma}\label{ptopi:lem} 
Let $\mu\in S$, where 
$S$ is the band defined in \eqref{properband:eq},  
and let $\delta=\delta(\mu)$ be as in Lemma \ref{exppmu:lem}. 
Then for all $x$, $y\in\mathbb{R}$ we have 
\begin{align}\label{ansatz:eq}
P_\mu(x,y)=\Pi_\mu(x,y)+R_\mu(x,y),
\end{align}
where $\Pi_\mu$ is as defined in \eqref{leadpmu:eq}, and 
\begin{align}\label{rmubound:eq}
R_\mu(x,y)=\mathcal{O}\big( (1+|x-y|^2)^{-1}\big),\ \forall 
x, y\in\mathbb R.
\end{align}
Moreover, $R_\mu(x,y)$, $\Pi_\mu(x, y)$ 
and $P_\mu(x, y)$ are continuous functions of 
$x, y\in\mathbb R$, and 
\begin{align}\label{pip:eq}
|P_\mu(x, y)| + |\Pi_\mu(x, y)| = 
\mathcal{O}\big( (1+|x-y|)^{-1}\big),\ \forall 
x, y\in\mathbb R.
\end{align}
If $\mu\notin \big(\sigma(H)\big)^\circ$, 
then $P_\mu(x, y)$ is a continuous function 
of $x, y\in\mathbb R$, and it satisfies the bound 
\begin{align}\label{pmu_edge:eq}
P_\mu(x, y) = 
\mathcal{O}\big( (1+|x-y|^2)^{-1}\big),\ \forall 
x, y\in\mathbb R.
\end{align}
\end{lemma}

\begin{proof} 
The continuity of the projection kernel $P_\mu(x, y)$ 
follows immediately from Lemma \ref{exppmu:lem}. If $\mu\notin \big(\sigma(H)\big)^\circ$, then 
\eqref{pmu_edge:eq} follows directly from \eqref{pmu_edgeS:eq}. 
 
Assume now that $\mu \in S$. 
Let $S_1, S_2, \dots,  S_N$, be ``genuine" spectral bands lying below the band $S$. 
%Thus, we can write  
 % \begin{align*}
  %P_\mu(x, y) = \sum\limits_{l=1}^N 
  %P_\mu[S_l](x, y) + P_\mu[S](x, y),
  %\end{align*}
%and hence, by Lemma \ref{exppmu:lem}, the kernel $P_\mu(x, y)$ is continuous 
%in $x, y\in\mathbb R$. Furthermore, if 
%$\mu\notin \big(\sigma(H)\big)^\circ$, 
%using \eqref{pmu_edgeS:eq} for each component, we get \eqref{pmu_edge:eq}.
%
%Suppose that $\mu \in S$ 
%Thus, 
%
Using Lemma \ref{exppmu:lem}, we can write  
  \begin{align*}
  P_\mu(x, y) = &\ \sum\limits_{l=1}^N 
  P_\mu[S_l](x, y) + P_\mu[S](x, y)\\
= &\ \Pi_\mu(x, y) + R_\mu(x, y), 
\end{align*}   
where
\begin{align*}
R_\mu(x, y) = \sum\limits_{l=1}^N P_\mu[S_l](x, y) + R_\mu[S](x, y).
\end{align*}
By Lemma \ref{exppmu:lem}, the kernel $R_\mu[S]$ and 
each term $P_\mu[S_l]$, $l = 1, 2, \dots, N$ 
satisfy \eqref{rmubound:eq}, 
whence \eqref{ansatz:eq}. 
The bound \eqref{pip:eq} for the kernel $P_\mu[S](x, y)$ 
follows from \eqref{pipS:eq}. 
\end{proof}

\section{Elementary Trace Norm Estimates}\label{sect:elemtrnest}

Throughout the proof of Theorem \ref{theorem} we need various trace class bounds for operators involved. 
It is interesting that for most of our needs we can get away with 
rather elementary bounds, as in \cite{LandauWidom}. 
This fact is due to the  specific 
form of the operators studied. 
As we see in the next few pages, many of the technical issues that we come across, 
boil down to trace class bounds for the operators of the form  
\begin{align}\label{IJK:eq}
\chi_I P_\mu \chi_J P_\mu \chi_K,
\end{align}
where $I, J, K\subset \mathbb R$ 
are some intervals that may depend on the parameter $\alpha>0$. 

%REMOVED
%We are interested in the 
%bounds that are uniform in $\alpha$.

\subsection{Schatten-von Neumann Classes}
Throughout this paper, we make use of the standard 
notation for the Schatten - von Neumann classes of operators 
$\mathfrak{S}_q, q>0$ in a Hilbert space, 
see e.g. \cite{BS}, \cite{Simon}. 
The class $\mathfrak{S}_q$ consists 
of all compact operators $A$  
whose singular values $(s_k(A))_{k\in\mathbb{N}}$ are $q$-summable, i.e.
\begin{align*}
\sum\limits_{k\in\mathbb{N}}s_k(A)^q <\infty.
\end{align*}
For $A\in \mathfrak{S}_q$ we denote by 
\begin{align*}
\|A\|_q:=
\bigg(\sum\limits_{k\in\mathbb{N}}s_k(A)^q\bigg)^{\frac{1}{q}},
\end{align*}
the norm (for $q\ge 1$) or quasi-norm (for $q\in (0, 1)$) 
on $\mathfrak S_q$. 
Note the ``H\"older's inequality"
\begin{align*}%\label{holder:eq}
\|AB\|_1\leq \|A\|_p \|B\|_q,
\quad\frac{1}{p} + \frac{1}{q} = 1,
\end{align*}
which holds for any $A\in\mathfrak S_p$ and $B\in \mathfrak S_q$. 
 While in this section we limit ourselves  
to estimates in the trace class $\mathfrak{S}_1$, 
Section \ref{sect:schattenest} treats 
operators in the classes $\mathfrak{S}_q$ for $q\in (0,1]$. 

The next elementary 
trace class estimate (see \cite[formula (12)]{LandauWidom}) 
plays a central role in our paper. We 
provide a proof for the reader's convenience.

\begin{lemma}\label{integralrankoneoperators}
Let $M\subset \mathbb{R}$ be a Borel-measurable set. Consider (weakly) measurable mappings $f, g:M\mapsto {\rm L}^2(\mathbb{R})$,
such that 
\begin{align*}
\int\limits_M \, \|f(z)\|_{{\rm L}^2} \|g(z)\|_{{\rm L}^2}\, dz <\infty.
\end{align*} 
Then the operator $A:{\rm L}^2(\mathbb{R})\to {\rm L}^2(\mathbb{R})$ 
which is defined via the form
\begin{align*}
\langle u,Av \rangle_{{\rm L}^2}:=\int\limits_M 
\langle u, f(z) \rangle_{{\rm L}^2} 
\langle g(z), v\rangle_{{\rm L}^2}\, dz, \ u,v\in {\rm L}^2(\mathbb{R}),
\end{align*} 
is of trace class with
\begin{align*}
\|A\|_1\leq
\int\limits_M \, \|f(z)\|_{{\rm L}^2} \|g(z)\|_{{\rm L}^2}\, dz
\end{align*}
\end{lemma}

\begin{proof}
Let $(d_n)_n$ and $(e_n)_n$ be orthonormal bases (ONB's) of 
${\rm L}^2(\mathbb{R})$ 
and denote by $\langle \cdot, \cdot \rangle$ and $\|\, \cdot\, \|$  
the scalar product and the norm respectively, 
on ${\rm L}^2(\mathbb{R})$. Then we have
\begin{align*}
\sum\limits_n 
\big|\langle d_n,Ae_n \rangle\big| 
&\leq \sum\limits_n \int_M  
\big|\langle d_n,f(z)\rangle\langle g(z), e_n\rangle\big|  \, dz\nonumber\\
&=\int_M 
\sum\limits_n\big|\langle d_n,f(z)\rangle\langle g(z), e_n\rangle\big| \, dz.
\end{align*}
The Cauchy-Schwartz inequality and Parseval's identity yield
\begin{align*}
\sum\limits_n\big|\langle d_n,f(z)\rangle
\langle g(z), e_n\rangle\big| 
\leq &\ 
\Big(\sum\limits_n \big|\langle d_n,f(z)\rangle\big|^2\Big)^{\frac{1}{2}}
\Big(\sum\limits_n \big|\langle g(z),e_n\rangle\big|^2\Big)^{\frac{1}{2}}
\\[0.2cm]
= &\ \|f(z)\|\, \|g(z)\|.
\end{align*}
This implies that 
\begin{align*}
\sum\limits_n 
\big|\langle d_n,Ae_n \rangle\big| 
\le \int\limits_M \, \|f(z)\|\, \|g(z)\|\, dz.
\end{align*}
The supremum of the left-hand side 
over all ONB's coincides with the trace norm, whence the claimed estimate. 
\end{proof}

Equipped with these basic trace norm estimates, we can 
start now our investigation of the operator \eqref{IJK:eq}.

\subsection{Replacing the  
Spectral Projection by its Approximation}
\label{sectionfirstorderapprox}

Let us recall the following general notation. 
If $f$, $g$ are real-valued functions we shall write 
$f\ll g$ if and only if $|f|\leq C |g|$ for some 
constant $C\geq 0$ which might depend on the potential 
$V$ but does not depend on the dilation parameter $\alpha$.  
Let $\Pi_\mu$ be as defined in Lemma \ref{exppmu:lem}.

\begin{lemma}\label{ptomu:lem}
Let $I$, $J$, $K\subset\mathbb{R}$ 
be intervals such that 
$I\cap J=\varnothing$ and $K\cap J = \varnothing$. 
Then we have
\begin{align}\label{ptomu:eq}
\big\|\chi_I P_\mu\chi_JP_\mu\chi_K 
- \chi_I\Pi_\mu\chi_J \Pi_\mu\chi_K\big\|_1\ll 1,
\end{align}
where the integral kernel of $\Pi_\mu$ is defined in \eqref{leadpmu:eq}. 
\end{lemma}

\begin{proof}
With the notation of Lemma \ref{ptopi:lem} we may write
\begin{align*}
\chi_IP_\mu\chi_JP_\mu\chi_K 
= \chi_I\Pi_\mu \chi_J \Pi_\mu\chi_K 
+ \chi_I \Pi_\mu \chi_J R_\mu\chi_K + \chi_IR_\mu\chi_J P_\mu\chi_K.
\end{align*}
Let us then estimate the trace norm of the 
operator $\chi_IR_\mu\chi_JP_\mu\chi_K$, which has the integral kernel
\begin{align*}
%(x,y)\mapsto 
 \chi_I(x) \chi_K(y)\int\limits_J R_\mu(x,z) P_\mu(z,y)\, dz.
\end{align*}
We apply Lemma \ref{integralrankoneoperators} with
\begin{align*}
f(x, z) = \chi_I(x)R_\mu(x,z),\ g(y, z)=\chi_K(y)
\conjugate{P_\mu(z,y)}=\chi_K(y)P_\mu(y,z),
\end{align*}
leading to
\begin{align*}
\|\chi_IR_\mu\chi_JP_\mu\chi_K\|_1
\le 
\int\limits_J \|R_\mu(\cdot,z)\|_{{\rm L}^2(I)}\|P_\mu(\cdot, z)\|_{{\rm L}^2(K)}\, dz.
\end{align*}
Thus estimates \eqref{rmubound:eq} and \eqref{pip:eq} yield
\begin{align*}
\big\|\chi_IR_\mu\chi_JP_\mu\chi_K\|_1&
\ll \int\limits_J \biggl[\int\limits_I  (1+|x-z|)^{-4}\, dx\biggr]^{\frac{1}{2}}
\biggl[\int\limits_K (1+|z-y|)^{-2}\, dy\biggl]^{\frac{1}{2}}\, dz\nonumber\\
&\ll \int\limits_J 
\big(1+\dist(z,I)\big)^{-{\frac{3}{2}}}
\big(1+\dist(z,K)\big)^{-\frac{1}{2}}\, dz 
\nonumber\\
&\ll \int\limits_J \big[\big(1+\dist(z,I)\big)^{-2}
+ \big(1+\dist(z,K)\big)^{-2}\big] \, dz
\ll 1.
\end{align*}
The operator $\chi_I \Pi_\mu\chi_JR_\mu\chi_K$ satisfies the same bound. 
Hence, the claim follows.
\end{proof} 

\subsection{Uniform Trace Norm Bounds}

Under particular assumptions on the intervals $I$, $J$ and $K$ 
the operator 
\eqref{IJK:eq} 
 is of trace class with uniformly bounded trace norm. 
 We list some of these conditions in the following proposition.

\begin{proposition}\label{proposition}
Let $I$, $J$, $K\subset \mathbb{R}$ be intervals such that one of the following condition holds:
\begin{enumerate}[label=(\roman*)]
\item \label{Ibounded} $|J|\ll 1$,
\item \label{lengthofIlowerorder}
%
%$\dist(I,J)$, $\dist(J,K)\gg 1$ and 
%
Either 
\begin{enumerate}
\item
$|J|\ll \max\{\dist(I,J),\ \dist(J,K)\}$, or
\item \label{KJ:item}
$|K| \ll \dist(J,K)$, $|I|\ll \dist(I,J)$, or
\item 
$|K| \ll \dist(J,K)$, $|J|\ll \dist(I,J)$.
\end{enumerate}
\item \label{IKoppositesidesofJ} 
$J$ is finite, and 
$I$ and $K$ lie on opposite sides of $J$, i.e. 
\begin{align} \label{IleqJleqK}
x\leq y\leq z \ \textup{or}\ \ 
 z\leq y\leq x,\ \ 
\textup{ for all } 
(x,y,z)\in I\times J\times K.
\end{align}
\item\label{Itbounded}
$|I|\ll 1$ and  $I\cap J = \varnothing, K\cap J = \varnothing$.
\end{enumerate}
Then the operator $\chi_I P_\mu \chi_J P_\mu \chi_K$ is 
uniformly bounded (independently of $\alpha$) in trace norm, i.e.
\begin{align}\label{tracenormuniformlybounded}
\|\chi_I P_\mu \chi_J P_\mu \chi_K\|_1\ll 1.
\end{align}
\end{proposition}

In the free case, i.e. for $V\equiv 0$, 
Proposition \ref{proposition} with assumptions similar to \eqref{Ibounded} and \ref{IKoppositesidesofJ} has been obtained in 
\citep[Lemma]{LandauWidom}.   

\begin{proof}[Proof of Proposition \ref{proposition}]
According to Lemma \ref{ptomu:eq} and bound \eqref{pip:eq},
\begin{align}\label{int:eq}
\|\chi_I P_\mu \chi_J P_\mu \chi_K\|_{\mathfrak{S}_1} 
\ll \int\limits_J 
\Biggl[\int\limits_I(1+|z-x|)^{-2}\, dx\Biggr]^{\frac{1}{2}}
\Biggl[\int\limits_K (1+|z-y|)^{-2}\, dy\Biggr]^{\frac{1}{2}}\, dz.
\end{align}
Let us estimate this integral under the conditions of the lemma.  

Assume condition \ref{Ibounded}. i.e. $|J|\ll 1$. 
Both integrals inside \eqref{int:eq} are uniformly bounded, 
even if $I$ and $K$ are 
unbounded. Thus the trace norm does not exceed $|J|\ll 1$, as required.  

Assume now condition \ref{lengthofIlowerorder}. 
Using 
the Cauchy-Schwarz inequality, we estimate the right-hand side 
of \eqref{int:eq} by 
\begin{align*}
\Biggl[\int\limits_J 
\int\limits_I(1+|z-x|)^{-2}\, dx dz\Biggr]^{\frac{1}{2}}
\Biggl[\int\limits_J 
\int\limits_K(1+|z-y|)^{-2}\, dy dz\Biggr]^{\frac{1}{2}}.
\end{align*}
The first integral is bounded by
\begin{align*}
|J|\big(1+\dist(I, J)\big)^{-1} \quad \textup{or} \quad 
|I|\big(1+\dist(I, J)\big)^{-1}.
\end{align*}
The second integral is bounded by
\begin{align*}
|K|\big(1+\dist(J, K)\big)^{-1}
\quad \textup{or} \quad 
|J|\big(1+\dist(J, K)\big)^{-1}.
\end{align*} 
Thus, under any of the conditions \ref{lengthofIlowerorder}, 
the right-hand side 
of \eqref{int:eq} is uniformly bounded, as required.   
 
Assume that the first of the 
conditions \eqref{IleqJleqK} holds. 
%We follow the same strategy as in the proof of 
%\citep[Lemma]{LandauWidom}.
%
%Using  Part \ref{Ibounded} of the lemma  
%we may assume that  
%$\dist(I,J)$,\ $\dist(J,K)\geq 1$. 
Let 
\begin{align}\label{intervals:eq}
I = (s_1, t_1), J = (s_2, t_2), K = (s_3, t_3)
\end{align}
with
\begin{align*}
-\infty \le s_1 < t_1\le s_2 < t_2 \le s_3 < t_3\le \infty.
\end{align*} 
Using \eqref{int:eq}, we get the bound
\begin{align*}
\|\chi_I P_\mu \chi_J P_\mu \chi_K\|_1 
&\ll \int\limits_{s_2}^{t_2} 
\Biggl[\int\limits_{s_1}^{t_1}|z-x|^{-2}\, dx\Biggr]^{\frac{1}{2}}
\Biggl[\int\limits_{s_3}^{t_3} |z-y|^{-2}\, dy\Biggr]^{\frac{1}{2}}\, dz
\nonumber\\
&\ll\int\limits_{s_2}^{t_2}
(z-t_1)^{-\frac{1}{2}}(s_3-z)^{-\frac{1}{2}}\, dz\\
&\le \int\limits_{s_2}^{t_2} 
(z-s_2)^{-\frac{1}{2}}(t_2-z)^{-\frac{1}{2}}\, dz 
= \int\limits_0^s 
z^{-\frac{1}{2}}(s-z)^{-\frac{1}{2}}\, dz,
\end{align*}
with $s = t_2 - s_2$.
By rescaling, the last integral equals
\begin{align*}
\int\limits_0^1 
z^{-\frac{1}{2}}(1-z)^{-\frac{1}{2}}\, dz\ll 1,
\end{align*}
which leads to \eqref{tracenormuniformlybounded} again.

Finally, assume that \ref{Itbounded} holds. 
The right-hand side of \eqref{int:eq} is bounded by 
\begin{align*}
  |I|^{\frac{1}{2}}\int\limits_J &\ \big(1+ \dist(z, I)\big)^{-1}
\big(1+ \dist(z, K)\big)^{-\frac{1}{2}} \, dz
\nonumber\\
&\ll 
\int\limits_J \big(1+ \dist(z, I)\big)^{-\frac{3}{2}}\, dz
+ \int\limits_J \big(1+ \dist(z, K)\big)^{-\frac{3}{2}}\, dz\ll 1.
\end{align*}
The proof is complete.
\end{proof}

\subsection{Replacing Almost Periodic Functions by their Mean Value} \label{almperbymean:subsec}

Looking at the formula 
\eqref{leadpmu:eq}  
we see that the kernel of 
$\chi_I \Pi_\mu \chi_J \Pi_\mu\chi_K$ 
contains  kernels of the form 
\begin{align}\label{prodker:eq}
S_{I,J,K}(x,y; f) = \chi_I(x) \chi_K(y)\int\limits_J 
\frac{f(z)}{(z-x)(z-y)}\, dz, 
\end{align}
where $f$ is a product of functions such as 
  $\Phi(\ \cdot\ , \delta)$ and $\conjugate{\Phi(\ \cdot\ , \delta)}$. 
The following lemma gives conditions for the intervals 
$I$, $J$, $K$ under which we may replace 
$f$ in $S_{I,J,K}(x,y; f)$ by its almost periodic 
mean value while the resulting error is uniformly bounded in trace norm.

\begin{lemma}\label{replacefunctionbymeanhankel}
Let $\Theta\subset\mathbb{R}$ be a 
countable set, and let $(a_\theta)_\theta\subset\mathbb{C}$ be  such that
\begin{align}\label{at:eq}
\sum\limits_{
\substack{\theta\in\Theta\\
\theta\not = 0}}|a_\theta|\big(1+|\theta|^{-1}\big)<\infty.
\end{align}
Let the function $f\in {\rm CAP}(\mathbb{R})$ 
be defined by
\begin{align*}
f(x)=\sum\limits_{\theta\in\Theta}a_\theta e^{i\theta x}.
\end{align*}
Assume that the intervals 
$I$, $J$, $K\subset \mathbb{R}$ satisfy 
$\ \dist(I,J),\ \dist(J,K)\gg 1$ and consider the operator $S_{I,J,K}(f)$ 
in ${\rm L}^2(\mathbb{R})$ with the integral kernel 
\eqref{prodker:eq}.
Then we have
\begin{align}\label{mean_int:eq}
\|S_{I,J,K}(f) - S_{I,J,K}\bigl(\mathcal M(f)\bigr)\|_1 \ll 1.
\end{align}
\end{lemma}

\begin{proof}
Without loss of generality 
we may assume that 
$\mathcal{M}(f)=0$, i.e. 
$0\notin \Theta$. (otherwise consider $f-\mathcal{M}(f)$). Consider the 
primitive $F(x):=\int_0^x f(t)dt$ of $f$. 
Then the assumption 
\eqref{at:eq} implies that $F$ is uniformly bounded:
\begin{align*}
|F(x)|=\Biggl|\sum\limits_{\theta\in\Theta}a_\theta
\int\limits_0^x e^{i\theta t} dt\Biggr| 
\leq \sum\limits_{\theta\in\Theta} \Big|\frac{a_\theta}{i\theta}(e^{i\theta x}-1)\Big|\ll 1,\ \forall x\in\mathbb{R}. 
\end{align*}
Let $J = (s, t)$, so integrating by parts gives
\begin{align}\label{SIJKafterIP}
S_{I,J,K}(x,y; f)&
=\left.\chi_I(x) \chi_K(y)\frac{F(z)}{(z-x)(z-y)}\right|_{z=s}^{t} 
\nonumber\\[1ex]
&\ \ \ 
+ \chi_I(x)\chi_K(y)\int\limits_J 
\Biggl[\frac{F(z)}{(z-x)^2(z-y)}+\frac{F(z)}{(z-x)(z-y)^2}\Biggr]
\, dz.
\end{align}
The first term in \eqref{SIJKafterIP} constitutes 
the kernel of a rank two operator, whose norm, and hence trace norm as well, 
is easily estimated 
by  a constant times $\dist(I, J)^{-1/2}\dist(J,K)^{-1/2}$. 
The second term on the right-hand side of \eqref{SIJKafterIP} 
is treated with the help of 
Lemma \ref{integralrankoneoperators}, 
as in the proof of Lemma \ref{ptomu:lem}. 
Thus \eqref{mean_int:eq} follows.
\end{proof}

\section{Schatten-von Neumann Class Estimates for 
Pseudo-differential Operators with Periodic 
Amplitudes}\label{sect:schattenest}

So far our main tool for getting trace-class estimates has been 
Lemma \ref{integralrankoneoperators}. 
At the final stages of the proof, however, when we pass to  
non-smooth functions $h$, we also need some estimates in 
more general Schatten-von Neumann classes $\mathfrak S_q$ with 
$q\in (0, 1]$. Lemma \ref{integralrankoneoperators} is not applicable 
any longer, and we have to appeal to 
other results available in the literature. 

We use the formalism of pseudo-differential operators ($\Psi$DO). For a 
complex-valued function 
$p = p(x, y, \xi)$, $x, y,\xi\in\mathbb R$, 
that we call \textit{amplitude}, we define the $\Psi$DO 
$\Op(p)$ that acts on
Schwartz class functions $u$ as follows:
\begin{align}\label{pdo:eq}
\Op(p)u(x) = \frac{1}{2\pi}
\iint e^{i\xi(x-y)} p(x, y, \xi) u(y)\, dy d\xi. 
\end{align}
This integral is well-defined, e.g. for any amplitude $p$ 
which is uniformly bounded and compactly supported in the variable $\xi$.

 The main result of this section is 
 the following proposition that implies Schatten-(quasi)norm 
 estimates for the operator 
 \begin{align}\label{aamu:eq}
 A_{\alpha,\mu}=B_{\alpha,\mu}(\mathds{1}-B_{\alpha,\mu})
 \end{align}
 (see Corollary \ref{corquasinorm}).

\begin{lemma}\label{NS:lem}
Let $I$, $\Omega\subset\mathbb{R}$ 
be bounded intervals, and let the function $p$ be 
${\rm C}^\infty$ in all three variables, 
$2\pi$-periodic in $x$ and $y$, and such that 
$p(x, y, \xi) = 0$ for all $x, y\in\mathbb R$, 
and $|\xi|\ge R$ with some $R>0$. 
Denote
\begin{align*}
p[\Omega](x, y, \xi) = p(x, y, \xi) \chi_\Omega(\xi).
\end{align*}
Then, for any $q\in (0,1]$ 
we have 
%
%there exists a constant $C_{q,p}$ 
%(also depending on $I$ and $\Omega$) such that 
\begin{align}\label{perom:eq}
\| \chi_{\alpha I}\Op(p)
(\mathds{1}-\chi_{\alpha I})\|_q\ll 1,
\end{align}
and 
\begin{align}\label{quasinormestpseudo}
\| \chi_{\alpha I}\Op\bigl(p[\Omega]\bigr)
(\mathds{1}-\chi_{\alpha I})\|_q \ll (\log\alpha)^{\frac{1}{q}}. 
\end{align}
The implicit constants in \eqref{perom:eq} and \eqref{quasinormestpseudo} 
depend on the amplitude $p$, number $R$  
and also on the intervals $I$ and $\Omega$. 

\end{lemma}

Our proof relies on similar results 
from \cite{Sobolev2014}. We state these results
in the form adjusted for our purposes.

\begin{proposition}\label{FA:prop}
Let $I$, $\Omega\subset\mathbb{R}$ 
be bounded intervals, and let the function $p = p(\xi)$ be 
${\rm C}^\infty_0(\mathbb R)$ 
with $p(\xi) = 0$ for $|\xi|\ge R$ with some $R>0$. 
For $q\in (0, 1]$ denote
\begin{align}\label{pm:eq}
N_q(p):= \max\limits_{0\le m\le \lfloor 2 q^{-1}\rfloor + 1}\sup_{\xi}|p^{(m)}(\xi)|
<\infty.
\end{align} 
Then
\begin{align}\label{FA1:eq}
\| \chi_{\alpha I}\Op(p)
(\mathds{1}-\chi_{\alpha I})\|_q\ll N_q(p),
\end{align}
and 
\begin{align} \label{FA:eq}
\| \chi_{\alpha I}\Op\bigl(p[\Omega]\bigr)
(\mathds{1}-\chi_{\alpha I})\|_q\ll (\log\alpha)^{\frac{1}{q}}
N_q(p).
\end{align} 
The implicit constants in \eqref{FA1:eq} and \eqref{FA:eq} 
depend on the intervals $I$, $\Omega$ and number $R$, but are independent 
of the amplitude $p$. 
\end{proposition}

Thus, our task is to extend Proposition \ref{FA:prop} to 
amplitudes, that are periodic in $x$ and $y$. 

A few remarks are in order. 
Proposition \ref{FA:prop} is a direct consequence of 
\cite[Corollary 4.4, Theorem 4.6]{Sobolev2014}. 
At this point it is important to emphasize that 
the main focus of \cite{Sobolev2014} was the \textit{quasi-classical} asymptotics, 
whereas our objective in the current paper is the \textit{scaling} asymptotics. 
In the context of pseudo-differential operators, these two types of asymptotics 
are equivalent if the amplitude $p$ is $x, y$-independent. 

\begin{proof}[Proof of Lemma \ref{NS:lem}] 
We prove only the bound \eqref{quasinormestpseudo}. The bound \eqref{perom:eq} 
can be derived in a similar way. 

Performing translations, dilations and renormalization of $\alpha$, 
one may assume that  
$I=\Omega=(0,1)$. Since $p$ is $2\pi$-periodic in $x$ and $y$, 
we can represent it as
\begin{align*}
p(x, y, \xi) = \sum\limits_{nl} e^{inx + il y}a_{nl}(\xi), 
\end{align*}
where $a_{nl}(\ \cdot\ )$ are ${\rm C}^\infty_0$ in 
$\xi$ with supports in $(-R, R)$, 
and decay in $n$ and $l$ 
faster than any reciprocal polynomial, uniformly in 
$\xi\in (-R, R)$. 
Precisely, 
a straightforward integration by parts shows that 
\begin{align*}
|a_{nl}^{(m)}(\xi)| \ll (1+|n|)^{-s}(1+|l|)^{-t}
\int\limits_{0}^{2\pi}\int\limits_{0}^{2\pi}
|\partial_x^s\partial_y^t\partial_\xi^m p(x, y, \xi)|\, dx dy, \quad 
n, l\in\mathbb Z,
\end{align*}
for arbitrary $t, s = 0, 1, \dots$, 
so that 
\begin{align*}
N_q(a_{nl})\ll (1+|n|)^{-s} (1+|l|)^{-t}, \quad n, l\in\mathbb Z,
\end{align*}
with a constant  independent of $n, l$, but depending on 
$s, t$, $q$ 
(see \eqref{pm:eq} for the definition of $N_q$). 
Consequently, the operator $\Op(p[\Omega])$ can be represented as follows:
\begin{align*}
\Op(p[\Omega]) = \sum\limits_{nl} e^{inx} A_{nl} e^{ily},\ \quad 
A_{nl} = \Op(a_{nl} \chi_{\Omega}).
\end{align*}
Using \eqref{FA:eq}, we immediately obtain the bound
\begin{align*}
\|\chi_{\alpha I}A_{nl}(\mathds{1}-\chi_{\alpha I})\|_q^q
\ll (1+|n|)^{-sq} (1+|l|)^{-tq}
\log\alpha.
\end{align*}
Employing the $q$-triangle inequality 
for the ideals $\mathfrak S_q$ (see  \cite[p. 262]{BS}), we  arrive at the bound
\begin{align*}
\|\chi_{\alpha I}\Op\bigl(p[\Omega]\bigr)
(\mathds{1}-\chi_{\alpha I})\|_q^q
\le &\ \sum\limits_{nl}
\|\chi_{\alpha I}A_{nl}(\mathds{1}-\chi_{\alpha I})\|_q^q\\
\ll &\ \log\alpha \sum\limits_{nl} (1+|n|)^{-sq} (1+|l|)^{-tq}.
\end{align*}
The sum on the right-hand side is finite if $sq, tq > 1$. 
This completes the proof. 
\end{proof}

\begin{corollary}\label{corquasinorm} 
Assume that $V\in {\rm C}^\infty(\mathbb R)$. 
Let $A_{\alpha, \mu}$ be as defined in \eqref{aamu:eq}.
\begin{enumerate}
\item
Let $I\subset \mathbb R$ 
be an interval. 
If $\mu\in \big(\sigma(H)\big)^{\circ}$, then for any $q\in (0, 1]$, 
\begin{align}\label{pmus:eq}
\|\chi_{\alpha I} P_\mu(\mathds{1}-\chi_{\alpha I})\|_q^q
\ll \log(\alpha). 
\end{align}
If $\mu\notin \big(\sigma(H)\big)^{\circ}$, 
then for any $q\in (0, 1]$, 
\begin{align}\label{pmu:eq}
\|\chi_{\alpha I} P_\mu(\mathds{1}-\chi_{\alpha I})\|_q^q
\ll 1. 
\end{align} 
\item 
For any $q\in (0,1]$,
\begin{align}\label{Aalphamuquasiest}
\|A_{\alpha,\mu}\|_q^q \ll
\begin{cases}
1,\ \mu\notin (\sigma(H))^\circ,\\[0.2cm]
\log(\alpha),\ \mu\in (\sigma(H))^\circ.
\end{cases}
\end{align} 
Moreover, assume that $h$ satisfies Condition 
\ref{h:cond}. 
Then $h(B_{\alpha,\mu})$ is of trace class and 
\begin{align}\label{bamu:eq}
\|h(B_{\alpha, \mu})\|_1 \ll 
\begin{cases}
\alpha |h(1)| + 1, \ \mu\notin (\sigma(H))^\circ,\\[0.2cm]
\alpha |h(1)| + \log(\alpha),\ \mu\in (\sigma(H))^\circ.
\end{cases}
\end{align}
\item\label{thmout:item}
If $\mu\notin (\sigma(H))^\circ$, then \eqref{trhb0:eq} holds. 
\end{enumerate}
The implicit constants in the inequalities 
\eqref{pmus:eq}, \eqref{pmu:eq}, \eqref{Aalphamuquasiest} 
and \eqref{bamu:eq} are independent of $\alpha$.
\end{corollary}

\begin{proof} 
It suffices to prove \eqref{pmus:eq} and \eqref{pmu:eq} 
for the projections $P_\mu[S]$ under the conditions 
$\mu\in S^\circ$ and $\mu\notin S^\circ$ respectively, for any 
``genuine'' band $S$ of the type \eqref{properband:eq}. 

Suppose that $\mu\in S^\circ$. 
By virtue of \eqref{calP:eq}, 
the operator $P_\mu[S]$ has the form 
$\Op\bigl(p[\Omega]\bigr)$ with 
\begin{align*} 
p(x, y, \xi) = E(x, \xi) \conjugate{E(y, \xi)} \quad
\textup{and}
\quad 
\Omega = (2k_j-\delta, \delta), 
\end{align*}
where $k_j$ is as defined in \eqref{kj:eq},  and 
$\delta\in (k_j, k_j+n/2)$ 
is the unique solution 
of the equation $\Lambda(\delta) = \mu$. 
The function $E(x, \xi)$ is $2\pi$-periodic in $x$, and due to 
the ${\rm C}^\infty$-smoothness of $V$, 
it is also ${\rm C}^\infty$-smooth in $x$. Now 
\eqref{pmus:eq} follows from \eqref{quasinormestpseudo}. 

Suppose that $\mu\notin S^\circ$. According to \eqref{projS:eq}, 
either $P_\mu[S] = 0$,
in which case \eqref{pmu:eq} is trivial, or 
\begin{align*}
P_\mu[S](x, y) = \int_{n \mathbb T} \Phi(x, k) \conjugate{\Phi(y, k)} dk. 
\end{align*}
Using a straightforward partition of unity on the circle $n\mathbb T$ 
one can represent $P_\mu[S]$ as a finite sum of operators 
of the form $\Op(p)$ with 
\begin{align*}
p(x, y, \xi) = E(x, \xi) \conjugate{E(y, \xi)} \zeta(\xi),\ 
\zeta\in {\rm C}^\infty_0(\mathbb R).
\end{align*}
Therefore, \eqref{pmu:eq} is a consequence of \eqref{perom:eq}.

From $\|P_\mu \chi_{(-\alpha, \alpha)}\|\leq 1$ we get that
\begin{align*}
\|A_{\alpha,\mu}\|_q 
= \|\chi_{(-\alpha, \alpha)} 
P_\mu (\mathds{1}-\chi_{(-\alpha, \alpha)}) 
P_\mu \chi_{(-\alpha, \alpha)}\|_q
\leq \|\chi_{(-\alpha, \alpha)} P_\mu (\mathds{1}-\chi_{(-\alpha, \alpha)})\|_q,
\end{align*}
and \eqref{Aalphamuquasiest} follows from \eqref{pmu:eq}. 

To prove \eqref{bamu:eq} we use the representation 
\eqref{decomph}:  $h(t) = t h(1) + h_1(t)$, so that 
$h_1(0) = h_1(1) = 0$ and $|h_1(t)|\ll t^q (1-t)^q$, where 
$q\in (0, 1]$ is the H\"older parameter of the function $h$. 
The first term on the 
right-hand side of \eqref{bamu:eq} follows from the bound
\eqref{balphabound:eq}. 
For the second term write:
\begin{align*}
\|h_1(B_{\alpha, \mu})\|_1 \ll \|A_{\alpha, \mu}^q\|_1
= \|A_{\alpha, \mu}\|_q^q,
\end{align*}
and hence the required bounds follow from \eqref{Aalphamuquasiest}. 
Together with Remark \ref{lead:rem}, this also implies 
Part (\ref{thmout:item}) of the corollary. 
\end{proof}
  
\section{Proof of Theorem \ref{theorem}: symmetric polynomial functions}\label{sect:proofpol}

By virtue of Corollary \ref{corquasinorm}(\ref{thmout:item}), 
the formula \eqref{trhb0:eq} is already proved. Thus it remains  
to prove Theorem \ref{theorem} for $\mu\in (\sigma(H))^\circ$. 
From now on we assume that $\mu$ is an interior point of a band $S$ of the type 
\eqref{properband:eq}.  As before, define 
$\delta\in (k_j, k_j + n/2)$ as 
the unique solution of the equation $\Lambda(\delta) = \mu$. 
For simplicity we abbreviate $\Phi=\Phi(\ \cdot \ ,\delta)$.

\subsection{Polynomial Classes}
We begin the proof of \eqref{trhb} with studying polynomial functions. 
The following classes of polynomials on the interval $[0,1]$ will be relevant:
\begin{align}\label{polspaces}
\mathfrak{P}&
:=\lbrace p:[0,1]\mapsto\mathbb{C}, \text{ polynomial}\rbrace,\nonumber\\
\mathfrak{P}_0&
:=\lbrace p\in\mathfrak{P}: p(0)=p(1)=0\rbrace, \nonumber\\
\mathfrak{P}_s&
:=\lbrace p\in\mathfrak{P}: p(t)=p(1-t)\text{ for all }t\rbrace, \nonumber\\
\mathfrak{P}_{s,0}&:=\mathfrak{P}_s\cap \mathfrak{P}_0.
\end{align}
As explained in Remark \ref{lead:rem}, it suffices to prove \eqref{trhb} for 
the functions $h_1$ satisfying Condition \ref{h:cond}, 
such that $h_1(0) = h_1(1) = 0$. 
Thus we need to study polynomials $p\in \mathfrak P_0$. In fact, it is enough 
to consider a basis 
of $\mathfrak{P}_0$. As in \citep{LandauWidom} we choose the basis
\begin{align*}
\lbrace
(p_n, q_n):  p_n(t)=(t(1-t))^n, \ 
q_n(t) = t(t(1-t))^n, \ n = 1, 2, \dots\rbrace,
\end{align*}
and start by considering the 
symmetric elements $p_n(t)$, 
which form a basis of $\mathfrak{P}_{s,0}$. 
So, we study the operators
\begin{align*}
p_n(B_{\alpha, \mu}) = A_{\alpha, \mu}^n,\ \ 
A_{\alpha,\mu} = B_{\alpha,\mu}(\mathds{1}-B_{\alpha,\mu}).
\end{align*}
In so doing, we follow the strategy of 
\cite{LandauWidom}, where 
a similar problem was analysed in the unperturbed 
case $V = 0$. In fact, our objective is to reduce the calculations to the 
unperturbed case, by using Lemmas \ref{mean:lem} and 
\ref{replacefunctionbymeanhankel}. 

\subsection{Trace Class Calculus for the Operator $A_{\alpha,\mu}$}

Rewrite the operator $A_{\alpha, \mu}$ in the form 
\begin{align*}
A_{\alpha,\mu} = A_{\alpha,\mu}^- + A_{\alpha,\mu}^+
\end{align*}
with
\begin{align}\label{apm:eq}
\begin{cases}
A_{\alpha,\mu}^-:&=\chi_{(-\alpha, \alpha)} 
P_\mu \chi_{(-\infty,-\alpha)} P_\mu \chi_{(-\alpha, \alpha)},\\[0.2cm]
A_{\alpha,\mu}^+:&= \chi_{(-\alpha, \alpha)} P_\mu \chi_{(\alpha,\infty)} 
P_\mu \chi_{(-\alpha, \alpha)}.
\end{cases}
\end{align}
Now we perform various transformations with each of these operators that 
constitute ``small" perturbations in $\mathfrak S_1$. 
Thus, it is natural to adopt the following notational 
convention: 

\begin{definition}
Let $A$ and $B$ be bounded operators on 
${\mathrm L}^2(\mathbb{R})$. 
We write $A\sim B$ 
if $\|A-B\|_{\mathfrak{S}_1}\ll 1$, 
uniformly in $\alpha\gg 1$. 
We write $A\approx B$ if $A$ and $B$ are 
trace class and $|\tr A - \tr B|\ll 1$ uniformly in $\alpha\gg1$. 
\end{definition}

Clearly, for 
trace class operators $A, B$ the relation 
$A\sim B$ implies $A\approx B$, but not 
the other way round. Note also, that $A\sim B$ implies $A^n\sim B^n$ 
for any $n = 1, 2, \dots$.

To begin with, by virtue of Proposition 
\ref{proposition}\ref{Ibounded},
\begin{align}\label{Aplusshort}
A_{\alpha,\mu}^+\sim
\chi_{(-\alpha, \alpha)} 
P_\mu\chi_{(\alpha+1,\infty)}
P_\mu \chi_{(-\alpha, \alpha)}.
\end{align}
and
\begin{align}\label{Aminusshort}
A_{\alpha,\mu}^-\sim \chi_{(-\alpha, \alpha)} 
P_\mu\chi_{(-\infty,-\alpha-1)}P_\mu
\chi_{(-\alpha, \alpha)}.
\end{align}

\subsubsection{Operators $D_\alpha^{\pm}$} 
The next step is to replace $A_{\alpha, \mu}^{\pm}$ 
with operators 
that do not contain any information on 
the function $\Phi(x, k)$. 
These are the operators 
$D_\alpha^{\pm}:{\mathrm L}^2(\mathbb{R})\mapsto {\mathrm L}^2(\mathbb{R})$, 
defined via their integral kernels
\begin{align*}
D_{\alpha}^+(x,y)& := \frac{1}{4\pi^2}
\chi_{(-\alpha,\alpha)}(x)\chi_{(-\alpha,\alpha)}(y)
\int\limits_{\alpha+1}^\infty 
\frac{1}{(z-x)(z-y)}\, dz,\nonumber\\
D_{\alpha}^-(x,y)&:= \frac{1}{4\pi^2}
\chi_{(-\alpha,\alpha)}(x)\chi_{(-\alpha,\alpha)}(y)
\int\limits_{-\infty}^{-\alpha-1} \frac{1}{(z-x)(z-y)}\, dz.
\end{align*}
Note that $D_\alpha^+$ and $D_\alpha^-$ are 
unitarily equivalent via the change $x\mapsto -x$. 
The crucial fact is that the asymptotic formulas for the traces of 
powers $(D_\alpha^{\pm})^n$ can be easily deduced from the results of 
\cite{LandauWidom}:

\begin{lemma}\label{LW:lem} 
Let $p_n(t) = t^n(1-t)^n$, $n = 1, 2, \dots$. Then  
\begin{align}\label{LW:eq}
\tr (D_{\alpha}^{\pm})^n 
= \frac{1}{4}\log \alpha\  \mathcal W(p_n) + o(\log\alpha),\ \alpha\to \infty,
\end{align}
where $\mathcal W(\ \cdot\ )$ is as defined in \eqref{whc}.
\end{lemma}

\begin{proof}
Since $D_\alpha^+$ and $D_\alpha^-$ are unitarily equivalent, 
we show 
\eqref{LW:eq} for $D_\alpha := D_\alpha^+$ only. 
By translation and reflection, 
the operator $D_\alpha$ is unitarily equivalent to 
the operator with kernel 
\begin{align*}
\frac{1}{4\pi^2}
\chi_{(1,2\alpha+1)}(x)\chi_{(1,2\alpha+1)}(y)
\int\limits_{0}^\infty 
\frac{1}{(z+x)(z+y)}\, dz,
\end{align*}
This is the kernel of the operator which is denoted by $K_c$ 
in \cite[p. 476]{LandauWidom}. Thus the formula 
\eqref{LW:eq} immediately follows from \cite[formula (19), p. 477]{LandauWidom}.
\end{proof}

A useful way to write $D_{\alpha}^\pm$ is 
\begin{align*}
D_\alpha^{\pm} = (Z_{\alpha}^{\pm})^\ast Z_\alpha^{\pm}, 
\end{align*}
where $Z_\alpha^\pm$ have kernels 
\begin{align}\label{Zalpha:eq}
Z_\alpha^+(x, y) 
= \frac{\chi_{(\alpha+1, \infty)}(x)
\chi_{(-\alpha,\alpha)}(y)}{2\pi(x-y)},
\quad \textup{and}\ \ 
Z_\alpha^-(x, y)= 
\frac{\chi_{(-\infty, -\alpha-1)}(x)
\chi_{(-\alpha,\alpha)}(y)}{2\pi(x-y)}
\end{align}
respectively. 
Now we need to establish a few facts for operators $D_\alpha^{\pm}$ 
and $Z_\alpha^{\pm}$.  Recall that we abbreviate 
$\Phi = \Phi(x, \delta)$, 
$\delta = \delta(\mu)$, 
remembering that $\mu$ is  strictly inside the band  
$S$.

\begin{lemma}\label{avw:lem} 
Denote by $Y_\alpha^\pm$ 
any of the two operators $Z_{\alpha}^\pm$ or $(Z_\alpha^{\pm})^*$.
With the notation as above,
\begin{align*}
Y_\alpha^{\pm} |\Phi|^2 (Y_\alpha^{\pm})^*
\sim \frac{1}{2\pi} Y_\alpha^{\pm}(Y_\alpha^{\pm})^*,\ \quad
Y_\alpha^{\pm} \Phi^2 (Y_\alpha^{\pm})^*
\sim 0. 
\end{align*}
\end{lemma}  
  
  \begin{proof} 
  We prove the lemma 
for the ``$+$" sign 
and for the case $Y_\alpha^+ = Z_\alpha^+$ only. The remaining cases 
are treated in the same way. 
For brevity we omit the superscript ``$+$" and write 
$Z_\alpha$ instead of $Z_\alpha^+$.  

The operator $Z_\alpha f Z_\alpha^\ast$ 
coincides with the operator $(4\pi^2)^{-1}S_{I, J, K}(f)$ 
with
\begin{align*}
I = K =(\alpha+1, \infty), 
J = (-\alpha, \alpha),
\end{align*} 
see the definition \eqref{prodker:eq}. Thus by Lemma 
\ref{replacefunctionbymeanhankel}, 
\begin{align*}
Z_\alpha f Z_\alpha^\ast
\sim \mathcal{M}(f) Z_\alpha Z_\alpha^\ast.
\end{align*}
In view of \eqref{mean:eq} and \eqref{mean0:eq}, 
$\mathcal M(|\Phi|^2)= (2\pi)^{-1}$ and $\mathcal M(\Phi^2) = 0$, whence 
the claimed result. 
\end{proof}
  
\begin{corollary}\label{kalphan:cor} 
Let
\begin{align}\label{kalphan:eq}
K_{\alpha, n}^\pm = 2\pi  
\big[
\Phi (D_\alpha^{\pm})^n \conjugate{\Phi}
+ \conjugate{\Phi} (D_\alpha^{\pm})^n \Phi \big],\ n = 1, 2, \dots.
\end{align}
Then for all $n = 1, 2, \dots$, we have
\begin{align}\label{powersofKalphaj}
(K_{\alpha, 1}^\pm)^n \sim  K_{\alpha,n}^{\pm},
\end{align}
and 
\begin{align}\label{trkalpha:eq}
(K_{\alpha, 1}^{\pm})^n \approx 2 (D_\alpha^{\pm})^n,\ \alpha\to\infty. 
\end{align}
\end{corollary}

\begin{proof}  For brevity we 
omit the superscript ``$\pm$" and write 
$K_{\alpha, 1}, D_\alpha$ instead of $K_{\alpha, 1}^{\pm}, D_\alpha^{\pm}$ etc. 
The powers of $K_{\alpha, 1}$ 
contain  terms of the form $D_\alpha f D_\alpha$ 
with $f = |\Phi|^2,\ \Phi^2$ or ${\conjugate\Phi}^2$. 
The operator $D_\alpha f D_\alpha$, 
is  written as
\begin{align*}
Z_\alpha^\ast Z_\alpha f Z_\alpha^\ast Z_\alpha.
\end{align*} 
Thus by Lemma \ref{avw:lem}, 
\begin{align*}
K_{\alpha, 1}^n \sim (2\pi)^n \big[
(\Phi D_\alpha \conjugate{\Phi})^n
+ (\conjugate{\Phi} D_\alpha \Phi)^n
\big]\sim 2\pi \big[
\Phi D_\alpha^n \conjugate{\Phi}
+ \conjugate{\Phi} D_\alpha^n \Phi 
\big], 
\end{align*}
as claimed.

In order to prove \eqref{trkalpha:eq}, 
use the cyclicity of the trace. 
If $n=1$, then, again by Lemma \ref{avw:lem},  
\begin{align*}
 \Phi D_\alpha \conjugate{\Phi} \approx  Z_\alpha |\Phi|^2 Z_\alpha^* \sim 
\frac{1}{2\pi} Z_\alpha  Z_\alpha^*
\approx \frac{1}{2\pi} D_\alpha.
\end{align*}
If $n\ge 2$, then, in the same way,
\begin{align*}
 \Phi D_\alpha^n  \conjugate{\Phi}
\approx Z_{\alpha}D_\alpha^{n-2} Z_\alpha^* Z_{\alpha }|\Phi|^2 Z_{\alpha}^*
\sim \frac{1}{2\pi} Z_{\alpha}D_\alpha^{n-2} Z_\alpha^* Z_{\alpha }Z_{\alpha}^*
\approx \frac{1}{2\pi} D_\alpha^{n}
\end{align*} 
The same is done with the component containing $\Phi$ and 
$\conjugate\Phi$ in the other order. This implies 
\eqref{trkalpha:eq}. 
Thus the proof is complete.
\end{proof}

\subsubsection{Approximating Operators $A_{\alpha, \mu}^{\pm}$}
Assume that $\mu$ is as before and $K_{\alpha, n}^{\pm}$ are 
as defined in \eqref{kalphan:eq}. 

\begin{lemma}\label{aalphapmkapm}
Let $S$ be a band of the spectrum of $H$, and let 
$\mu\in S^\circ$. 
Let $\delta\in (k_j, k_j + n/2)$ be the unique solution of the equation 
$\Lambda(\delta) = \mu$. Then we have
\begin{align}\label{atok:eq}
(A_{\alpha,\mu}^\pm)^n\sim (K_{\alpha, 1}^\pm)^n,
\end{align}
and
\begin{align}\label{aaplus2:eq}
A_{\alpha, \mu}^n 
\sim (A_{\alpha,\mu}^+)^n + (A_{\alpha,\mu}^-)^n 
\sim (K_{\alpha, 1}^{+})^n + (K_{\alpha, 1}^{-})^n.
\end{align}
for every $n = 1, 2, \dots$.
\end{lemma}
 
\begin{proof}
To prove \eqref{atok:eq} it suffices to consider the case $n=1$. 
As before, we do it for $A_{\alpha, \mu}^+$ only, omitting the 
superscript ``$+$".  From \eqref{Aplusshort} 
and Lemma \ref{ptomu:lem} it follows that 
\begin{align*}%\label{api:eq}
A_{\alpha,\mu}^+ 
\sim \chi_{(-\alpha, \alpha)} \Pi_\mu 
\chi_{(\alpha+1,\infty)}\Pi_\mu\chi_{(-\alpha, \alpha)}.
\end{align*}
By \eqref{leadpmu:eq} and \eqref{Zalpha:eq},
\begin{align*}
\chi_{(\alpha+1, \infty)} \Pi_\mu \chi_{(-\alpha, \alpha)}
= &\ - 2\pi i \big( \Phi Z_\alpha \conjugate\Phi 
-  \conjugate\Phi  Z_\alpha \Phi 
\big),\\
\chi_{(-\alpha, \alpha)} \Pi_\mu 
\chi_{(\alpha+1, \infty)}
= &\ 2\pi i
\big( \Phi Z_\alpha^* \conjugate\Phi 
-  \conjugate\Phi  Z_\alpha^* \Phi 
\big),
\end{align*}
so that 
\begin{align*}
A_{\alpha,\mu}^+ 
\sim &\ 
4\pi^2\big(
\Phi 
Z_\alpha^* |\Phi|^2 Z_\alpha 
\conjugate{\Phi} + 
\conjugate{\Phi} Z_\alpha^* |\Phi|^2 Z_\alpha {\Phi}\big)\\
&\ - 4\pi^2 \big(
{\Phi}Z_\alpha^* {\conjugate\Phi}^2 Z_\alpha \Phi + 
\conjugate{\Phi}Z_\alpha^* \Phi^2 Z_\alpha \conjugate\Phi\big).
\end{align*}
Consequently, by Lemma \ref{avw:lem},
\begin{align*}
A_{\alpha,\mu}^+ 
\sim 
2\pi\big({\Phi} Z_\alpha^* Z_\alpha \conjugate{\Phi} + 
\conjugate{\Phi} 
Z_\alpha^* Z_\alpha {\Phi}\big) = K_{\alpha, 1},
\end{align*}
as required. 

Proof of \eqref{aaplus2:eq}. 
By the definition \eqref{apm:eq}, 
\begin{align*}
A_{\alpha,\mu}^-A_{\alpha,\mu}^+ 
= \chiab P_\mu \Big(\chi_{(-\infty,-\alpha)}
P_\mu 
\chi_{(-\alpha, \alpha)}  
P_\mu \chi_{(\alpha,\infty)}\Big)P_\mu
\chi_{(-\alpha, \alpha)}.
\end{align*}
By Proposition \ref{proposition}(\ref{IKoppositesidesofJ}), 
the trace norm of the operator in the middle is uniformly bounded, and hence 
$A_{\alpha,\mu}^-A_{\alpha,\mu}^+\sim 0$. In the same way one checks that 
$A_{\alpha,\mu}^+A_{\alpha,\mu}^-\sim 0$. Thus 
\begin{align*}
A_{\alpha,\mu}^n\sim (A_{\alpha,\mu}^+)^n + (A_{\alpha,\mu}^-)^n,
\end{align*}
and \eqref{aaplus2:eq} is now a consequence of 
\eqref{atok:eq}.  
\end{proof}

\subsection{Proof of Theorem \ref{theorem} for Symmetric Polynomials}
 By \eqref{aaplus2:eq}, \eqref{trkalpha:eq} and \eqref{LW:eq}, 
 \begin{align}\label{powernhankeleq}
 \tr A_{\alpha, \mu}^n = &\ \tr (K_{\alpha}^+)^n + 
 \tr  (K_{\alpha}^-)^n + O(1)\notag\\[0.2cm]
 = &\ 2\tr (D_{\alpha}^+)^n + 
 2\tr  (D_{\alpha}^-)^n + O(1)\notag\\[0.2cm]
= &\ \log (\alpha) \mathcal W(p_n) + o(\log(\alpha)),\ n=1, 2, \dots.
 \end{align}
Hence, Theorem \ref{theorem} for polynomials 
$p\in \mathfrak P_{s, 0}$ follows from the identity
$p_n(B_{\alpha, \mu}) = A_{\alpha, \mu}^n$. 
\qed

\section{Proof of Theorem 
\ref{theorem}: conclusion}\label{sect:proofconcl}

As above, we assume that $\mu\in S^\circ$, where $S$ is a band of the type 
\eqref{properband:eq}.

\subsection{Arbitrary Polynomials}\label{ap:subsect}
So far we have proved Theorem \ref{theorem} 
for polynomials $p\in\mathfrak{P}_{s,0}$ 
(cf. \eqref{polspaces} for notation). 
To extend this result to arbitrary 
$p\in\mathfrak{P}_0$ 
 it remains to treat basis elements of the form 
 $q_n(t) = t[t(1-t)]^n$, $n = 1, 2, \dots$. 
Following \citep{LandauWidom} for the free case, 
 this is done by a symmetry argument 
 that reduces 
 $\tr\big[B_{\alpha,\mu}A_{\alpha,\mu}^n\big]$ to 
 $\tr A_{\alpha,\mu}^n$. 
 
\begin{lemma}\label{btoone:lem}
For every $n = 1, 2, \dots,$ we have 
\begin{align}\label{btoone:eq}
B_{\alpha,\mu}\big(A_{\alpha,\mu}\big)^n
\approx \frac{1}{2}\tr\big(A_{\alpha,\mu}\big)^n, 
\end{align}
as $\alpha\to\infty$.
\end{lemma}

Compared to \cite{LandauWidom}, the proof requires some extra work. 
The main difference is that instead of the reflection symmetry 
used in \cite{LandauWidom}, we use the translation symmetry of the operators. 
The operators $A_{\alpha, \mu}^+$ and $A_{\alpha, \mu}^-$(see 
\eqref{apm:eq}) are considered separately. Applying Proposition \ref{proposition}
(\ref{KJ:item}), we get 
\begin{align}\label{Aeq:eq}
A_{\alpha, \mu}^+\sim \chi_{(-\alpha, \alpha)} 
P_\mu \chi_{(\alpha,3\alpha)}P_\mu  \chi_{(-\alpha, \alpha)}.
\end{align}
Let $U_\alpha^{\pm}$ be 
the unitary shift operators defined by 
\begin{align*}
U_\alpha^{\pm} f(x) = f(x \mp \alpha_0),\ 
\alpha_0 = 2\pi \left\lfloor \frac{\alpha}{2\pi}
\right\rfloor. 
\end{align*}
The equivalence \eqref{Aeq:eq} implies that 
\begin{align}\label{atopret:eq}
(U_\alpha^+)^* A_{\alpha, \mu}^+
U_\alpha^+\sim \chi_{(-2\alpha, 0)}
P_\mu\chi_{(0,2\alpha)}P_\mu\chi_{(-2\alpha, 0)}.
\end{align}
Indeed, \eqref{Aeq:eq} yields:
\begin{align*}
(U_{\alpha}^+)^* A_{\alpha, \mu}^+U_\alpha^+
\sim \chi_{(-\alpha-\alpha_0, \alpha-\alpha_0)}
P_\mu\chi_{(\alpha-\alpha_0,3\alpha-\alpha_0)}
P_\mu\chi_{(-\alpha-\alpha_0, \alpha-\alpha_0)},
\end{align*}
since $(U_{\alpha}^+)^* P_\mu U_{\alpha}^+ = P_\mu$. 
Now, to get \eqref{atopret:eq}, one needs to use 
repeatedly Proposition \ref{proposition}\ref{Ibounded}, \ref{Itbounded}.
 We denote 
\begin{align*}
\chi_\alpha^+ = \chi_{(0, 2\alpha)},\ 
\chi_\alpha^- = \chi_{(-2\alpha, 0)}
\end{align*}
and 
\begin{align*}
T_{\alpha,\mu}^{\pm}:=\chi_\alpha^{\mp}P_\mu\chi_\alpha^{\pm}P_\mu\chi_\alpha^{\mp}.
\end{align*}
Thus one can write
\begin{align}\label{atot:eq}
(U_\alpha^{\pm})^* A_{\alpha,\mu}^{\pm}U_\alpha^{\pm}
\sim T_{\alpha,\mu}^{\pm}.
\end{align}
This relation with the ``$+$" sign coincides with 
\eqref{atopret:eq}, and for the ``$-$" sign it is proved in the same way. 
The proof of Lemma \ref{btoone:lem} begins with the following observation.  
 
\begin{lemma}\label{auxiliarylemmaasymmetric}
For any $n = 1, 2, \dots,$ we have
\begin{align}\label{ptoq:eq}
P_\mu (T_{\alpha,\mu}^\pm)^n\approx  
(\mathds{1}-P_\mu)(T_{\alpha,\mu}^\mp)^n, \ \text{as } \alpha\to\infty.
\end{align}
\end{lemma}

\begin{proof}
For brevity we write 
$\chi^{\pm}=\chi_\alpha^{\pm}$, 
$T^{\pm}=T_{\alpha,\mu}^{\pm}$, $P=P_\mu$, and 
$Q = \mathds{1}-P$. We have
\begin{align}\label{cycle:eq}
P(T^+)^n &=P\chi^-P\chi^+P\chi^-(T^+)^{n-1}
= -P\chi^-Q\chi^+P\chi^-(T^+)^{n-1}\nonumber\\
&= P(\mathds{1}-\chi^-)Q\chi^+P\chi^-(T^+)^{n-1}\nonumber\\
&= P\chi^+Q\chi^+P\chi^-(T^+)^{n-1} + R_1 + R_2,
\end{align}
with
\begin{align}
R_1&= P\chi_{(2\alpha,\infty)}Q\chi^+P\chi^-(T^+)^{n-1},\nonumber\\
R_2&= P\chi_{(-\infty,-2\alpha)}Q\chi^+P\chi^-(T^+)^{n-1}\nonumber.
\end{align}
We notice that $Q = \mathds{1}-P$ can 
be replaced by $-P$ in $R_1$. 
By Proposition \ref{proposition}\ref{IKoppositesidesofJ}, 
\begin{align*}
\chi_{(2\alpha, \infty)}P\chi^+P\chi^-\sim 0,
\end{align*}
 so that $R_1\sim 0$. To handle $R_2$, observe that 
 \begin{align}\label{plustominus:eq}
 \chi^+P\chi^-(T^+)^{n-1} = (T^-)^{n-1}\chi^+P\chi^-,
 \end{align}
and hence, by cyclicity of the trace, 
\begin{align*}
R_2\approx Q\chi^+ (T^-)^{n-1} \chi^+P\chi^- P\chi_{(-\infty,-2\alpha)}.
\end{align*} 
 Applying Proposition \ref{proposition}\ref{IKoppositesidesofJ} 
 to the factor  $\chi^+P\chi^- P\chi_{(-\infty,-2\alpha)}$ 
 we infer that $R_2\approx 0$.

Apply \eqref{plustominus:eq} to the first operator on the 
right-hand side of \eqref{cycle:eq} and use again the cyclicity:
\begin{align*}
P\chi^+Q\chi^+P\chi^-(T^+)^{n-1}
= &\ P\chi^+Q (T^-)^{n-1}\chi^+P\chi^-\\
\approx &\ Q (T^-)^{n-1} \chi^+P\chi^- P\chi^+ 
= Q (T^-)^n.
\end{align*}
Together with \eqref{cycle:eq} this yields \eqref{ptoq:eq} for the ``$+$" 
sign. 
The relation \eqref{ptoq:eq} for the ``$-$" sign is obtained in the same way. 
\end{proof}

\begin{proof}[Proof of Lemma \ref{btoone:lem}] 
We shall use the simplified notation as 
in the proof of Lemma \ref{auxiliarylemmaasymmetric} and 
also write 
$A=A_{\alpha,\mu}$, 
$A^\pm=A_{\alpha,\mu}^\pm$, and 
$B=B_{\alpha,\mu}$. 
First observe that $BA^n \approx PA^n$. Thus by 
\eqref{aaplus2:eq} and \eqref{atot:eq}, 
\begin{align*}
BA^n\approx P (A^+)^n + P (A^-)^n 
\approx P (T^+)^n + P (T^-)^n.
\end{align*}
By Lemma \ref{auxiliarylemmaasymmetric}, 
\begin{align*}
2 P (T^{\pm})^n \approx P (T^{\pm})^n + (\mathds{1}-P) (T^{\mp})^n,
\end{align*}
so that 
\begin{align*}
2P (T^+)^n + 2P (T^-)^n \approx &\ 
P (T^{+})^n + (\mathds{1}-P)(T^{-})^n +  
P (T^{-})^n + (\mathds{1}-P) (T^{+})^n\\[0.2cm]
= &\ (T^{+})^n + (T^{-})^n.
\end{align*}
Using \eqref{atot:eq} and \eqref{aaplus2:eq} again, we get 
\begin{align*}
2BA^n\approx A^n, 
\end{align*}
which leads to \eqref{btoone:eq}, and hence 
completes the proof.  
\end{proof}

As a consequence of Lemma \ref{btoone:lem},  
Theorem \ref{theorem} can be proved for arbitrary $p\in\mathfrak{P}_0$.

\begin{proof}[Proof of Theorem \ref{theorem} for arbitrary polynomials] 
It remains to prove the theorem for polynomials of the form 
$q_n(t) = t p_n(t)$, $n = 1, 2, \dots$. 
From Lemma \ref{btoone:lem} and 
\eqref{powernhankeleq} we deduce that  
\begin{align}\label{qn:eq}
\tr\big[B_{\alpha,\mu}\big(A_{\alpha,\mu}\big)^n\big] =
\frac{1}{2}\log(\alpha) \mathcal W(p_n) + o(\log(\alpha)),\ \alpha\to\infty.
\end{align}
To convert $\mathcal W(p_n)$ into $\mathcal W(q_n)$ we perform a very 
elementary calculation:
\begin{align*}
\pi^2\mathcal W(q_n) = 
\int_0^1 \frac{tp_n(t)}{t(1-t)}\, dt
=  \int_0^1  \frac{p_n(t)}{1-t}\, dt
= \int_0^1  \frac{p_n(t)}{t}\, dt.
\end{align*}
Therefore
\begin{align*}
2\pi^2 \mathcal W(q_n)
= \int_0^1 p_n(t)\biggl(\frac{1}{1-t}+ \frac{1}{t} \biggr)\, dt
= \int_0^1  \frac{p_n(t)}{t(1-t)}\, dt
= \pi^2\mathcal W(p_n).
\end{align*}
Together with \eqref{qn:eq} this leads to Theorem \ref{theorem} 
for arbitrary polynomials $p\in \mathfrak P_0$.
\end{proof}

\subsection{Closure of the Asymptotics}
 \label{closure:subsect}
Throughout this final section we 
assume that $h$ satisfies Condition \ref{h:cond}. 
The proof splits into three steps.

\underline{Step 1.} 
First we prove the theorem for continuous functions $h$ 
such that $h(0) = h(1) = 0$ that are 
differentiable at $t =0$ and $t = 1$. 
Without loss of generality 
we may assume that $h$ is real-valued 
(otherwise treat real and imaginary part separately). 
The differentiability condition at $t=0$ and $t=1$ 
implies that $h(t)=t(1-t)g(t)$ 
for a continuous real-valued function $g$. 
Fix $\epsilon>0$. Due to the Stone-Weierstrass theorem, 
there exist a real-valued polynomial 
$p\in\mathfrak{P}$ with $\|p-g\|_\infty<\epsilon$. 
Denoting $\tilde{p}(t):=t(1-t)p(t)$ we estimate
\begin{align}\label{hleqpolynomial}
h(t)\leq t(1-t)(p(t)+\epsilon)=\tilde{p}(t)+\epsilon t(1-t), 
\end{align}
and 
\begin{align}\label{hgeqpolynomial}
h(t)\geq t(1-t)(p(t)-\epsilon)=\tilde{p}(t)-\epsilon t(1-t).
\end{align}
The monotonicity of the trace in combination 
with \eqref{hleqpolynomial} gives
\begin{align*}
\tr\big[ h(B_{\alpha,\mu})\big] \leq \tr\big[\tilde{p}
(B_{\alpha,\mu})\big]+\epsilon\tr\big[B_{\alpha,\mu}(\mathds{1}-B_{\alpha,\mu})\big].
\end{align*}
From Theorem \ref{theorem} for polynomials from $\mathfrak P_0$, we get 
%\begin{align*}
%\frac{\tr\big[ h(B_{\alpha,\mu})\big]}{\log(\alpha)}
%\leq \mathcal{W}(\tilde{p})
%+ \epsilon \mathcal{W}(t(1-t)) +o(1),\ \alpha\to\infty.
%\end{align*}
%Thus, we arrive at
\begin{align*}
\limsup\limits_{\alpha\to\infty}
\frac{\tr\big[ h(B_{\alpha,\mu})\big]}{\log(\alpha)}
\leq \mathcal{W}(\tilde{p})+\epsilon \mathcal{W}(t(1-t))
= \mathcal W(\tilde p) + \frac{\epsilon}{\pi^2},
\end{align*}
where we have used 
that $\mathcal W(t(1-t)) = \pi^{-2}$, see \eqref{whc}. 
Moreover, we notice that 
\begin{align*}
\big|\mathcal{W}(h)-\mathcal{W}(\tilde{p})\big| = 
\big|\mathcal{W}(h-\tilde{p})\big|\leq \frac{\epsilon}{\pi^2},
\end{align*}
and hence,
\begin{align*}
\limsup\limits_{\alpha\to\infty}
\frac{\tr\big[ h(B_{\alpha,\mu})\big]}{\log(\alpha)}
\leq \mathcal{W}(h) +  \frac{2\epsilon}{\pi^2}.
\end{align*}
In the same way \eqref{hgeqpolynomial} implies
\begin{align*}
\liminf\limits_{\alpha\to\infty}\frac{\tr\big[ h(B_{\alpha,\mu})\big]}{\log(\alpha)}\geq \mathcal{W}(h) - \frac{2\epsilon}{\pi^2},
\end{align*}
and as $\epsilon>0$ was chosen arbitrarily we 
deduce \eqref{trhb} for our choice of $h$.

\underline{Step 2.}
Now let $h$ be a continuous function, 
which is H\"older-continuous at $0$ and $1$ 
with exponent $q\in(0,1]$, so that 
%In particular, there exists a constant $C_h>0$ such that
\begin{align*}
|h(t)|\ll  t^q(1-t)^q,\ t\in[0,1].
\end{align*}
Fix again $\epsilon >0$ and choose a 
smooth function $\zeta_\epsilon$ such that $0\leq \zeta_\epsilon\leq 1$ and 
\begin{align*}
\zeta_\epsilon(t)
=\begin{cases}
1, &t\in [0,\epsilon/2]\cup [1-\epsilon/2,1],\nonumber\\
0, &t\in [\epsilon,1-\epsilon].
\end{cases} 
\end{align*}
In view of the estimate
\begin{align*}
|(\zeta_\epsilon h)(t)|\ll [t(1-t)]^q\zeta_\epsilon(t)
\ll \epsilon^r [t(1-t)]^r,\ r = \frac{q}{2},
\end{align*}
we have 
\begin{align*}
\|(\zeta_\epsilon h)(B_{\alpha,\mu})\|_1
\ll \epsilon^r
\|B_{\alpha,\mu}(\mathds{1}-B_{\alpha,\mu})\|_r^r.
\end{align*}
By Corollary \ref{corquasinorm}, the right-hand side does not exceed 
$\log(\alpha)$, $\alpha \ge 2$.  
Consequently, 
\begin{align}\label{esttrzetaepsh}
\frac{\big|\tr\big[(\zeta_\epsilon h)
(B_{\alpha,\mu})\big]\big|}{\log(\alpha)}\ll \epsilon^r,\ \alpha\geq 2.
\end{align}
On the other hand, $h_\epsilon = (1-\zeta_\epsilon)h$ 
vanishes in a vicinity of $0$ and $1$ and, therefore, by Step 1, 
we have
\begin{align}\label{esttrheps}
\tr\big[h_\epsilon(B_{\alpha,\mu})\big] 
= \log(\alpha)\mathcal{W}(h_\epsilon)+o(\log(\alpha)), \ \alpha\to\infty.
\end{align}
It is clear that
\begin{align}\label{estwheps}
\mathcal{W}(h)-\mathcal{W}(h_\epsilon)
\ll \biggl(\int_0^\epsilon + \int_{1-\epsilon}^1\biggr) 
t^{q-1}(1-t)^{q-1}\, dt
\ll \epsilon^q.
\end{align} 
Combining 
\eqref{esttrzetaepsh}, 
\eqref{esttrheps}, and \eqref{estwheps} gives
\begin{align*}
\limsup\limits_{\alpha\to\infty}\Big|\frac{\tr(h(B_{\alpha,\mu}))}{\log\alpha}-\mathcal{W}(h)\Big|\ll \epsilon^r.
\end{align*}
Since $\epsilon>0$ is arbitrary, this yields the claim. 

\underline{Step 3.}
Suppose that $h$ satisfies Condition \ref{h:cond}. 
Let $t_0\in (0, 1)$ be a point such that $h$ is continuous 
on $[0, t_0]$ and $[1-t_0, 1]$. 
Fix an $\epsilon >0$. 
Then one can find 
two continuous functions $h_1, h_2$ as at Step 2, such that 
\begin{align*}
h_1(t) = h_2(t) = h(t),\ t\in [0, t_0]\cup[1-t_0, 1],\ \quad 
h_1(t)\le h(t)\le h_2(t),\ t\in [0, 1],
\end{align*}
and $\|h_2-h_1\|_{\rm L^1}<\epsilon$. 
This implies that 
\begin{align*}
|\mathcal W(h_1) - \mathcal W(h)|, |\mathcal W(h_2)-\mathcal W(h)|
\ll \epsilon.
\end{align*}
Now, in view of monotonicity, we have
\begin{align*}
\tr h_1(B_{\alpha, \mu})\le \tr h(B_{\alpha, \mu})
\le \tr h_2(B_{\alpha, \mu}).
\end{align*}
Thus, by Step 2, 
\begin{align*}
\limsup\limits_{\alpha\to\infty}
\left|\frac{\tr h(B_{\alpha, \mu})}{\log(\alpha)} - \mathcal W(h)\right|
 \ll \epsilon.
\end{align*}
Since $\epsilon>0$ is arbitrary, 
the required result follows.  

%\

%----------------------------------------------------------------------------------------
%	BIBLIOGRAPHY
%----------------------------------------------------------------------------------------
\bibliographystyle{beststyle}
\bibliography{bibliography}
%\printbibliography[heading=bibintoc]
%----------------------------------------------------------------------------------------
\Addresses
\end{document}